\documentclass [11pt] {article} 

\usepackage[T1]{fontenc}
\usepackage[english] {babel}
\usepackage{amssymb}
\usepackage{amsmath}
\usepackage{amsthm}
\usepackage{bbm}
\usepackage{xcolor}
\usepackage{graphicx}
\usepackage{subcaption}
\usepackage{hyperref}

\newtheorem{thm}{Theorem}[section]
\newtheorem{prop}[thm]{Proposition}
\newtheorem{lem}[thm]{Lemma}
\newtheorem{rem}[thm]{Remark}

\title{{Spectrally accurate fully discrete schemes for some nonlocal and nonlinear integrable PDEs via explicit formulas}}

\author{
Yvonne Alama Bronsard%
\thanks{Massachusetts Institute of Technology, Cambridge, MA 02139, 
\texttt{yvonneab@mit.edu}}, \  
Xi Chen%
\thanks{Department of Mathematics and Computer Science, University of Basel, \texttt{xi01.chen@unibas.ch}}, \ 
Matthieu Dolbeault%
\thanks{
Institute for Geometry and Practical Mathematics, RWTH Aachen, 
\texttt{dolbeault@igpm.rwth-aachen.de}}
}
\date{}
\begin{document}

\baselineskip=15pt

\providecommand{\keywords}[1]{{\small\textit{Keywords---}} #1}

\maketitle

\begin{abstract} 
We construct fully-discrete schemes for the Benjamin–Ono, Calogero–Sutherland DNLS, and cubic Szeg\H o  equations on the torus, which are {\it exact in time} with {\it spectral accuracy} in space.
 We prove spectral convergence for the first two equations, of order $K^{-s+1}$ in $L^2$ norm for initial data in $H^s(\mathbb  T)$, $s>1$, with 
an error constant depending {\it linearly} on the final time instead of exponentially.
These schemes are based on {\it explicit formulas}, which have recently emerged in the theory of nonlinear integrable equations.
Numerical simulations show the strength of the newly designed methods both  at short and long time scales, thanks to the remarkable fact that the computational cost of the method is independent of the final time.
These schemes open doors for the understanding of the long-time dynamics of integrable equations.
\end{abstract}


\section{Introduction}
We consider fully discrete approximations to three nonlinear and nonlocal integrable equations.
Important progress has recently been made on the theoretical level for these equations, opening the way to new numerical approaches that we present here.

The first equation, central in the theory of integrable systems, is the Benjamin--Ono equation
\begin{equation}
\label{eq:BO}\tag{BO}
\partial_t u(t,x) = \partial_x \left( |D| u - u^2 \right)(t,x), \quad  u_{|t=0}(x)= u_{0}(x), \quad (t,x) \in \mathbb{R} \times \mathbb{T}, 
\end{equation}
\renewcommand*{\theHequation}{notag.\theequation}%
where $u(t,x)\in \mathbb{R}$ is a real-valued solution, $D = \frac{1}{i}\frac{d}{dx}$ and $|D|$ is defined in Fourier space as
\[
\widehat{|D|f}(k) = |k|\widehat{f}(k), \quad f\in L^2(\mathbb T).
\]
This nonlocal quasilinear dispersive equation introduced by Benjamin \cite{BO} (see also Davis--Acrivos~\cite{DA67} and Ono \cite{Ono}) models long, unidirectional internal gravity waves in two-layered fluids, as rigorously justified in the recent work of Paulsen \cite{P24}. Although the \eqref{eq:BO} equation resembles the well-known Korteweg--de\,Vries equation (KdV), with Airy's dispersive flow $\partial_t+\partial_{xxx}$ replaced by a Schr\"odinger-type flow $\partial_t - \partial_x  |D|$, the dispersion present in the equation is significantly reduced, thus rendering the control of the derivative in the nonlinearity a harder problem. 
Using techniques from the theory of integrable systems, and notably a Birkhoff normal form transformation, Gérard--Kappeler--Topalov \cite{GKT} show global well-posedness of \eqref{eq:BO} in $H^s(\mathbb{T})$ spaces if $s>-\frac{1}{2}$ and ill-posedness otherwise, see also Killip--Laurens--Vi\c san \cite{KLV24}. For a recent survey of known results and open challenges we refer to the book of Klein--Saut \cite[Chapter 3]{KS21}, and the references therein.
\newline

The second equation considered is the focusing ($+$ sign) or defocusing ($-$ sign) Calogero--Sutherland derivative nonlinear Schr\"odinger (DNLS) equation
\begin{equation}
\label{eq:CS}\tag{CS}
i\partial_t u + \partial_x^2u \pm \frac{2}{i} u\partial_x \Pi(|u|^2) = 0,\quad  u_{|t=0}(x)= u_{0}(x), \quad (t,x) \in \mathbb{R} \times \mathbb{T}, 
\end{equation}
\renewcommand*{\theHequation}{notag2.\theequation}%
where the Riesz--Szeg\H o projector $\Pi$ is defined in Fourier space as 
\begin{equation}
\label{eq:Pi}\tag{$\Pi$}
\widehat{\Pi f}(k) = \mathbbm{1}_{k\ge 0} \, \widehat{f}(k), \quad f\in L^2(\mathbb T).
\end{equation}
\renewcommand*{\theHequation}{notag3.\theequation}%
This is a nonlocal nonlinear Schr\"odinger equation and is derived from the Calogero--Sutherland--Moser system in \cite{Ca71, Su71, Su75}. This physical model represents a system of $N$ identical particles interacting pairwise. Abanov, Bettelheim and Wiegmann  \cite{ABW} formally show that taking the thermodynamic limit of such a model and applying a change of variables leads
to the \eqref{eq:CS} equation. One can also recover \eqref{eq:CS} formally as a limit of the intermediate nonlinear Schr\"odinger equation introduced by Pelinovsky \cite{Pe95}. Badreddine \cite{Bad24} achieves global well-posedness in the Hardy--Sobolev space $H_{+}^s(\mathbb{T}) = \Pi(H^s)$ for $s\geq 0$, 
by additionally requiring small initial data $\| u_{0}\|_{L^{2}}^2=\frac{1}{2\pi}\int_{-\pi}^{\pi}|u_0|^2 < 1$ in the focusing case.   
Remarkably, even though \eqref{eq:CS} is a completely integrable equation, one can expect the existence of {\it finite time blow-ups} in the focusing case. 
Indeed, on the real line, Gérard--Lenzmann \cite{GL22} prove global well-posedness in $H^1_+(\mathbb{R})$ if $\|u_0\|_{L^2(\mathbb R)}^2=\int_{\mathbb R}|u_0|^2\leq 2\pi$, whereas Kim--Kim--Kwon \cite{KKK24} very recently show the existence of smooth solutions with mass arbitrarily close to $2\pi$, whose $H^1$ norm blows up in finite time. In the periodic setting, the dynamics of the focusing \eqref{eq:CS} equation for initial data with mass greater or equal to one remains a compelling open problem.
\newline

Finally, the third equation is the cubic Szeg\H o equation
\begin{equation} \label{eq:S}\tag{S}
i\partial_t u =\Pi(|u|^2u),\quad  u_{|t=0}(x)= u_{0}(x), \quad (t,x) \in \mathbb{R} \times \mathbb{T}, 
\end{equation}
\renewcommand*{\theHequation}{notag4.\theequation}
where $\Pi$ is once again the Riesz--Szeg\H o projector.
This equation is
introduced in \cite{GG10} by Gérard and Grellier who show global well-posedness in $H_{+}^s(\mathbb{T})$ for $s\geq {1}/{2}$, using the fact that the norm $H^{{1}/{2}}_{+}$ is conserved. 
As opposed to the last two equations, \eqref{eq:S} is non-dispersive, and is used as a toy model for studying the NLS equation when there is a lack of dispersion properties due to the confining geometry of the domain. Another motivation comes from the study of {\it wave turbulence}, since the equation admits energy cascades from low to high frequencies, as well as 
energy transfers from high to low frequencies due to the almost time-periodicity of the solution~\cite{GG17}.
\newline

A key feature of these integrable equations is the existence of Lax pairs \cite{GGKM67, Lax68}, from which an infinite number of conservation laws can be derived.
Recently, ground-breaking results were obtained for the \eqref{eq:BO}, \eqref{eq:CS} and \eqref{eq:S} equations,
proving the existence of an {\it explicit formula} for the solution $u$, based upon their Lax pair structure.  
On the torus, the first result is due to G\'erard and Grellier \cite{GG15} for the \eqref{eq:S} equation, followed by G\'erard \cite{G} for the \eqref{eq:BO} equation, and Badreddine \cite{Bad24} for the \eqref{eq:CS} equation.
The goal of this paper is to make a bridge between these new analytical results and the field of computational mathematics, 
by obtaining efficient approximations to the above equations via these explicit formulas and proving their convergence on the discrete level.

Before presenting our methodology, we
discuss previous numerical discretizations to the above equations. While \eqref{eq:CS} and \eqref{eq:S} are relatively recent, there exists a vast literature on the numerical approximation of \eqref{eq:BO}. We detail some of these works here, with an emphasis on results providing explicit convergence rates. Given the nonlocal nature of the linear operator $\partial_x|D|$ and its diagonal expression $ik|k|$ in the Fourier variables, pseudo-spectral methods are usually adopted 
due to their computational efficiency.
This leads to spatial semi-discretizations, 
which are then coupled with suitable time approximations, such as finite differences.
For a comparison of different efficient spectral numerical methods we refer to Boyd--Xu \cite{BX11}, and to Deng--Ma \cite{DM09-2} for a semi-discrete pseudo-spectral error analysis result.
  In the fully discrete case, Pelloni--Dougalis
\cite{PD01} prove convergence of a scheme combining leap-frog in time and spectral Galerkin method in space,
whose error analysis is refined in Deng--Ma \cite{DM09},
while Galtung \cite{G16} studies a Crank--Nicolson Galerkin scheme.
On the full line $\mathbb{R}$, fully discrete approximations are also analyzed, where the authors consider a large torus in numerical implementations. We refer to Thomee--Murthy
\cite{T98} and Dutta--Holden--Koley--Risebro \cite{DHK16} for a finite difference approximation, and Dwivedi--Sarkar \cite{DS24} for a local discontinuous Galerkin method.

Unlike previous methods, which rely on discretizing the underlying PDE, we introduce novel schemes based on the {\it explicit formulas} of \cite{G, Bad24,GG15}. Although these formulas give an explicit representation of the solution $u(t)$ in terms of the initial data $u_0$ and the time $t$, they involve taking the inverse of a product of nonlocal operators, whose manipulation and computation are far from obvious, see equations \eqref{eq:explicit-BO-disc}, \eqref{eq:explicit-CS-disc} and \eqref{eq:explicit-S-disc}. We hence propose a different path and derive from these non-trivial formulas a simpler representation of the solution which is suitable to implement in Fourier space, see equations \eqref{eq:explicitForm}, \eqref{eq:explicitForm2} and \eqref{eq:explicitForm3}. 
Remarkably, while the \eqref{eq:BO}, \eqref{eq:CS}, and  \eqref{eq:S} equations are {\it nonlinear}, these explicit formulas only involve {\it linear} operators in the unknown (for a fixed initial data $u_0$), which we then compute in the same way one would solve a linear PDE via Fourier transforms.

From these formulas we construct schemes which are {\it exact in time} with {\it spectral accuracy} in space and whose computational cost (CPU) is {\it independent} of the final time. Namely, by playing closely with the Lax-pair formulation and explicit formula we show that the new schemes for \eqref{eq:BO} and \eqref{eq:CS} converge in $H^r$, $r\geq 0$, at the rate $K^{-s+1+r}$ with $u_0\in H^s$, $s>1$. The complexity to compute the solution at any time $t$ is in $\mathcal O(K^3)$, where $K$ is the number of Fourier modes in the discretization.
This allows for an extremely accurate and efficient approximation, surpassing the methods in the literature, as we detail below. First, our fully-discrete convergence result generalizes previous ones. Indeed, our schemes converge to the solution
at arbitrary high order, for sufficiently high regularity $s$.
In contrast, all fully-discrete convergence result in the literature that obtain explicit rates \cite{T98, PD01, DM09, G16} are solely second order in time.
Moreover, these works require at least the initial data to belong to  $H^s$  with $ s \ge 5$, and we reduce this regularity requirement to $ s > 1$.
In addition, our result holds in the error norm $H^r$ for any $r \ge 0$, whereas previous works focus on either the $L^2$ norm \cite{T98,PD01,G16} or the energy norm $H^{1/2}$ \cite{DM09}.
Secondly, we introduce a completely different approach for proving global error bounds, which greatly improves on prior error analysis results,
both at {\it short times} $t=\mathcal O(1)$ and {\it long times} $t\gg 1$.
Indeed, even over short times $t=\mathcal O(1)$, the balance between precision and computational cost for our scheme beats that of any previous result, as discussed in Section 4.2. 
We next detail the benefits over long times.

On the theoretical level, the error constant in our convergence result grows {\it linearly} in the final time $t$, see Figure~\ref{fig:lin-err}.
This is to be compared with previous error analysis results, which combine stability and local error bounds with a Gronwall-type argument to obtain convergence of the method with an error constant which grows {\it exponentially} in $t$, see Figure~\ref{fig:exp-growth-RK4}. To the best of our knowledge, this is the first error analysis result for a nonlinear dispersive equation on bounded domains with a sharp error constant depending linearly on the final time instead of exponentially, when no smallness assumptions on the initial condition are imposed. We refer to Remark~\ref{rem:longtime} for a discussion on the subject.
On the computational level, thanks to the fact that the scheme is exact in time, the CPU cost is independent of the final time $t$, while all time-stepping schemes have a complexity proportional to $t$, see Figure~\ref{fig:errCst-CPU}.

The proposed schemes are hence perfectly fit for simulating the long-time behavior of these PDEs, which open doors for the understanding of the global well-posedness \cite{Bad24}, soliton resolution \cite{IT19, KK24}, small dispersion limits \cite{Bad24.3, Ga23.1, BGGM24-2}, and norm inflation or blow-up phenomena  \cite{GG17, BE22, HK24, KKK24}.

\begin{rem}[Long-time error analysis]\label{rem:longtime}
An important step
towards long time estimates was made by Carles--Su \cite{CS22b}. Using scattering theory in order to obtain quantitative time decay estimates, they show uniform in time error estimates for the nonlinear Schr\"odinger equation on the full space $\mathbb{R}^d$, for a Lie splitting discretization. Their convergence analysis is, however, limited to  $\mathbb{R}^d$ as it heavily relies on dispersive decay estimates, which do not hold on the torus $\mathbb{T}^d$ or more generally on compact domains. 

Our work addresses this limitation, by presenting a convergence result on the torus $\mathbb{T}$, 
with an error constant depending linearly on the final time $t$. We make this possible by heavily exploiting the integrable nature of the equation, which allows us to go from a nonlinear problem, to a linear representation of the solution. 
Unlike in the case of the full space $\mathbb{R}$, 
we expect the error to accumulate linearly over time, and in this sense the result presented here is sharp.
\end{rem}

\begin{rem}[Extension to other PDEs]
Much progress is currently being made in the theory of nonlinear integrable equations thanks to the explicit formulas, see for instance \cite{Bad24.3,Bad24.2, BGGM24-1, BGGM24-2}. This motivates the search of such formulas for different PDEs. We refer for example to  the very recent advances on the half-wave maps equation \cite{GL24} and hyperbolic one-dimensional conservation laws~\cite{C25}. The methods provided in this work should be adaptable to other PDEs once their explicit formula has been established, or to perturbations of PDEs for which explicit formulas exist. Since the \eqref{eq:BO}, \eqref{eq:CS}, and \eqref{eq:S} equations arise as asymptotic regimes of more complex systems, the present approach could support a strategy for constructing schemes for these systems by treating them as perturbations of the limiting integrable models.
For instance, the Benjamin--Ono equation is derived from the {\it internal water wave system} by taking appropriate limits in the surface tension $\sigma$, the shallowness parameter $\mu$, and the nonlinearity parameter $\epsilon$; see~\cite{P24}.
In such cases, our scheme may serve as a zeroth-order approximation to the full system, complemented by classical methods to integrate the higher-order corrections.
\end{rem}

\begin{rem}[On the role of Hardy spaces in deriving explicit formulas]
The holy grail would be to have explicit formulas for the one-dimensional cubic NLS or KdV equations. Currently, a key structural requirement for deriving explicit formula is that the PDE admits a Lax‐pair formulation well‐defined on the Hardy space $L_{+}^2(\mathbb{T})$, which is not the case for either equation, see \cite[Chp. 2.4.1 and Chp. 2.3.1]{KS21} for the definitions of their respective Lax pairs.
As a result, no explicit formula is presently available -- though identifying an appropriate formula for these fundamental equations remains an active and attractive research direction.
\end{rem}

\begin{rem}[From scattering transforms to explicit formulas]
A substantial body of work, encompassing both formal and rigorous results, has been devoted to the theoretical and numerical resolution of completely integrable Hamiltonian systems on $\mathbb{R}$  via the inverse scattering transform (IST) and its discretization.
The most prominent examples include the KdV and one-dimensional NLS equations, along with their two-dimensional integrable generalizations, the Kadomtsev–Petviashvili and Davey–Stewartson equations. The IST procedure can be viewed as a nonlinear change of variables and is typically decomposed into three steps: the direct scattering transform, the time evolution of the spectral data, and the inverse scattering transform. We refer to the book of Klein--Saut \cite{KS21} for a recent overview covering these topics.
In contrast, for \eqref{eq:BO} posed on the circle $\mathbb{T}$, the analogue of the IST is given by the existence of Birkhoff coordinates --- or action-angle variables --- which was established by Gérard--Kappeler--Topalov~\cite{GKT}. 

The key advantage of the explicit formula and the associated discretized scheme presented here is that they bypass the need for performing direct or inverse scattering transforms. The solution $u(t)$ can be expressed explicitly in terms of the initial data $u_0$, and the scheme is simply computed in the Fourier basis.
\end{rem}

\subsection{Results}
\label{sec:results}
Let $K$ be the number of Fourier frequencies used in the discretization. Using symmetry arguments, we only need to work with non-negative frequencies $k = 0, \dots, K-1$. By analogy with \eqref{eq:Pi}, we define the truncated projector $\Pi_K$ in Fourier space as
\[
\widehat{\Pi_K f}(k) = \mathbbm{1}_{0\le k<K} \, \widehat{f}(k), \quad f\in L^2(\mathbb T).
\]
The new fully discrete spectral schemes $u_K$ for \eqref{eq:BO}, \eqref{eq:CS} and \eqref{eq:S} are essentially obtained by substituting every occurrence of $\Pi$ by $\Pi_K$  in the explicit formulas \eqref{eq:explicitForm}, \eqref{eq:explicitForm2} and \eqref{eq:explicitForm3}. Written in Fourier variables, the schemes are of the form
\begin{equation}
\label{eq:newScheme}
\smallskip
\widehat{u_K}(t, k) = {\bf e}_0\cdot \left(e^{-it{\bf M}}e^{it{\bf A}}{\bf S}^*\right)^k e^{-it{\bf M}}\,{\bf u}_0, \quad k = 0, \dots, K-1,
\end{equation}
with matrices ${\bf M,A, S^*}\in \mathbb C^{K\times K}$ defined in equations \eqref{AM-BO}, \eqref{AM-CS} and \eqref{AM-S}, and vectors 
\[
{\bf e}_0=(1,0,\dots,0)\quad\text{and}\quad {\bf u}_0 =(\widehat{u_0}(k))_{0\leq k < K}.
\]
For negative frequencies $k = -(K-1), \dots, -1$, we take
$\widehat{u_K}(t, k) = \overline{\widehat{u_K}(t, -k)}$ in the case of \eqref{eq:BO}, and $\widehat{u_K}(t, k) =0$ for the other two equations.

Our main convergence result for \eqref{eq:BO} is given below.
\begin{thm}\label{thm:main}
For every $s>1$ and $u_0 \in H^s(\mathbb{T})$ real-valued, let $u \in \mathcal{C}(\mathbb{R}, H^s(\mathbb{T}))$ be the unique solution of~\eqref{eq:BO}. For $K\in \mathbb N$, let $u_{K}$ be the numerical scheme \eqref{eq:newScheme}, in the case of \eqref{eq:BO}. Then for any $t>0$ and $r \in [0, s] $, there exists an explicit constant $C>0$, given in equation \eqref{eq:cstFinal}, depending only on $s$, $\|u_0\|_{H^s(\mathbb{T})}$ and $\|u(t)\|_{H^s(\mathbb{T})}$ such that 
\begin{equation}
\label{eq:cv}
\|u(t) - u_{K}(t)\|_{H^r} \leq C(1+t)K^{-s+1+r}.
\end{equation}
\end{thm}

\begin{figure}[ht]
    \centering
    \begin{subfigure}[t]{0.48\textwidth}
        \centering
        \includegraphics[height=5.3cm]{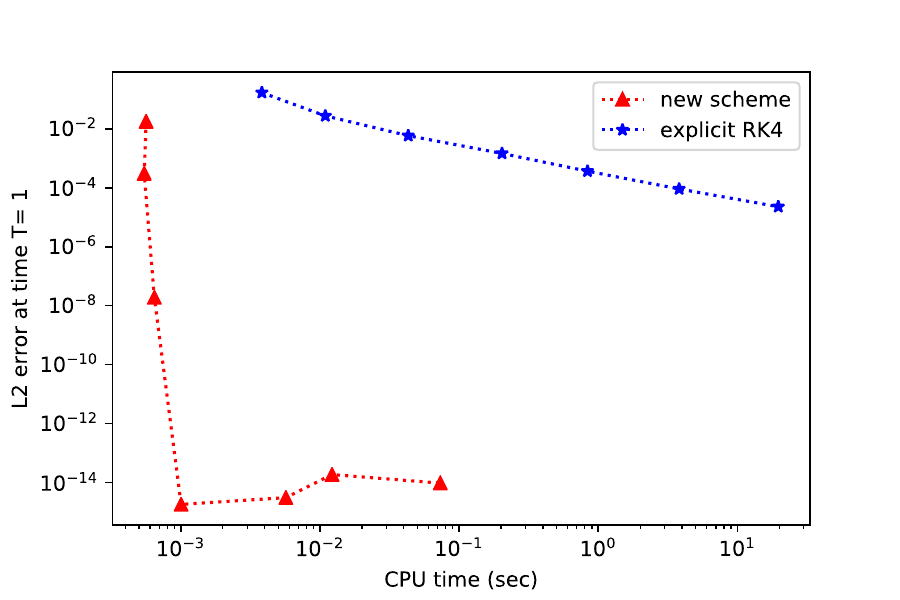}\caption{}
        \label{short-CPU}
    \end{subfigure}%
    ~ 
    \begin{subfigure}[t]{0.5\textwidth}
        \centering
        \includegraphics[height=5.3cm]{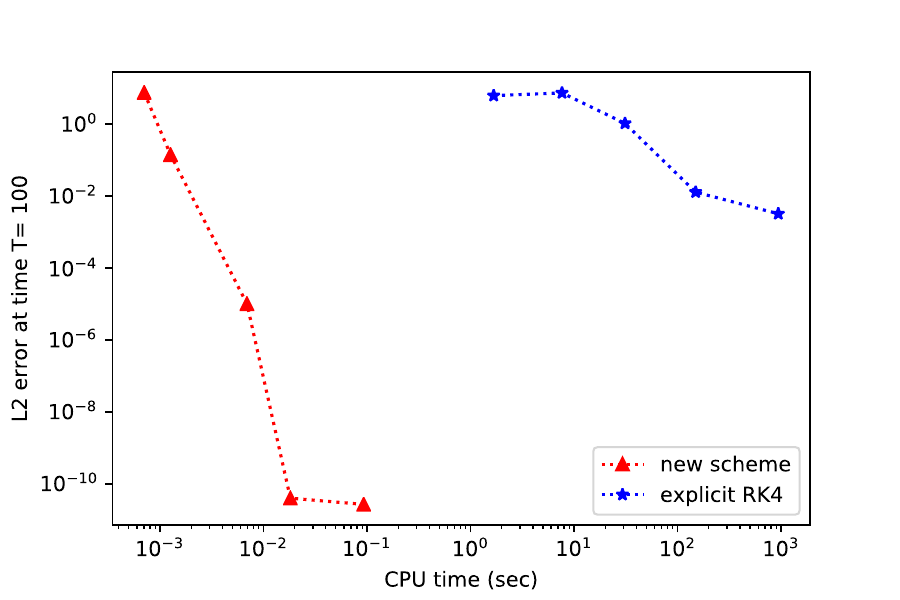}\caption{}\label{long-CPU}
    \end{subfigure}
\caption{ Left: convergence plot for the \eqref{eq:BO} equation in $L^2$ against the computational cost at time $T=1$ for the exact solution \eqref{eq:exactSmooth} (with $c=\frac{15}{4\pi}$). Each point corresponds to a different value of $K$, ranging between $2^3$ and $2^9$. The new scheme in red is given in equation \eqref{eq:newScheme}, the scheme in blue is the Fourier pseudo-spectral method coupled with a standard Runge-Kutta method (RK4). Right: solution computed up to $T=100$ for \eqref{eq:exactSmooth} with $c=\frac{15}{\pi}$, with $K$ a power of two ranging between $2^5$ to $2^9$.}
\end{figure}

We start by making a few remarks on Theorem~\ref{thm:main}.
As previously discussed, our theorem, together with the efficient complexity of the schemes, generalizes and improves all prior results in the literature, as further detailed in Section~\ref{sec:comparison}. In particular, concerning our spectral rates, they coincide with those obtained  in the literature when analyzing semi-discrete Fourier pseudo-spectral methods, see Deng--Ma \cite{DM09-2} in the case of \eqref{eq:BO} (with $r=1/2$ and smooth enough solutions \footnote[1]{Among other assumptions, the authors require $u \in C([0,T], H^\alpha)$ and $\partial_t u \in C([0,T], H^\alpha)$ with $\alpha \ge 2$. Using the PDE to convert temporal derivatives into spatial ones, this boils down to imposing at least $s\ge 4$.}) and Maday--Quarteroni \cite{MQ88} in the case of the KdV equation (with $r=1$ and $s>4$). In an earlier work, Deng--Ma \cite{DM09} obtain
a fully discrete scheme with a better spatial rate $K^{-s+r}$ (with $r=1/2$ and very smooth solutions \footnote[2]{The authors require $u \in C^1([0,T], H^2)$, $\partial_t^2 u \in C([0,T], H^1)$ and $\partial_t^3 u \in L^2([0,T], H^{5/2})$, which using the PDE and its global well-posedness, is equivalent to imposing $s\ge 17/2$.}), but for a spectral Galerkin scheme with significantly higher computational cost, which offsets the gain in accuracy and leads to their improved work \cite{DM09-2}. We detail in Remark~\ref{rem:read-rates} how the rate $K^{-s+1+r}$ is natural given our schemes.

\begin{figure}[ht]
\centering
\includegraphics[height=5.5cm]{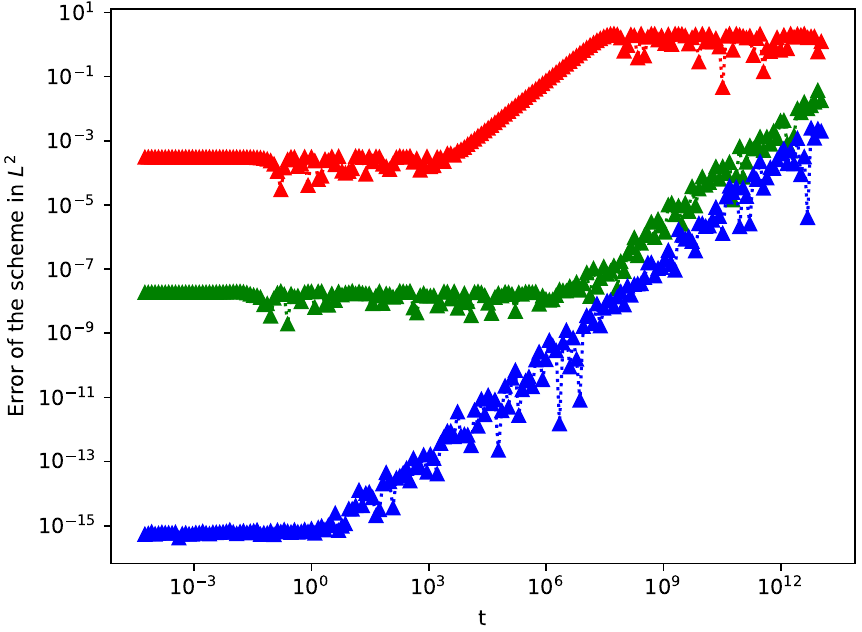}
\caption{Linear growth over time of the $L^2$ error of our scheme \eqref{eq:newScheme} for \eqref{eq:BO} with $K = 16$ (red), $K=32$ (green) and $K=64$ (blue).}
\label{fig:lin-err}
\end{figure}

Lastly, the error constant $C = C(s,\|u_0\|_{H^s}, \|u(t)\|_{H^s})$ in \eqref{eq:cv} is a non-decreasing function of its arguments, as can be seen in equation \eqref{eq:cstFinal}. Note that it depends only on $\|u_0\|_{H^s}$ and $\|u(t)\|_{H^s}$, instead of $\|u\|_{L^\infty([0,t],H^s)}$,
because we do not compute the solution at intermediate times, unlike any time-stepping method. 
In the case of \eqref{eq:BO}, the solution remains uniformly bounded in $H^s$ at all times \cite{GKT,KLV24}, and the dependence of $\|u\|_{L^\infty(\mathbb R,H^s)}$ on $\|u_0\|_{H^s}$ can be made explicit when $s$ is integer,
see Remark~\ref{cor:norm_of_u(t)}. Hence, one could remove $\|u(t)\|_{H^s}$ from the statement of Theorem~\ref{thm:main}, up to a change in the constant $C$. 
 \newline

A similar convergence result holds for \eqref{eq:CS}, except in the focusing case with critical or supercritical mass $\|u_0\|_{L^2(\mathbb{T})} \geq 1$, for which the global existence of the solutions is not known.
\begin{thm}\label{thm:CS}
For every $s>1$ and $u_0 \in H^s(\mathbb{T})$ with $\|u_0\|_{L^2(\mathbb{T})} <1$ in the focusing case, let $u \in \mathcal{C}(\mathbb{R}, H^s(\mathbb{T}))$ be the unique solution of~\eqref{eq:CS}. Under the same assumptions as in Theorem~\ref{thm:main}, the
 scheme \eqref{eq:newScheme}, in the case of \eqref{eq:CS}, converges as follows
\begin{equation*}
\|u(t) - u_{K}(t)\|_{H^r} \leq C(1+t)K^{-s+1+r}.
\end{equation*}
\end{thm}
\bigskip

The \eqref{eq:S} equation, being non-dispersive, differs significantly from its companions \eqref{eq:BO} and \eqref{eq:CS}. In particular, our proof of convergence crucially relies on establishing an equivalence of norms between Sobolev spaces and powers of the Lax operators, see Lemma~\ref{lem:LuplusC2}. This equivalence of norms holds thanks to the presence of the derivative operator $D$ in the definition of the Lax operators, which is absent in the case of the \eqref{eq:S} equation. As a result, proving convergence for the numerical scheme approximating \eqref{eq:S} remains an interesting open problem, and would require a completely different approach from the one developed here for \eqref{eq:BO} and \eqref{eq:CS}. Nevertheless, from a computational perspective, the scheme \eqref{eq:newScheme} for \eqref{eq:S} appears well suited for efficiently exploring long-time dynamics, such as the phenomenon of norm inflation \cite{GG10, GG17}, as its computational cost remains independent of the final time; see Section~\ref{sec:comparison}.

\begin{rem}[Reading the convergence rate from the scheme]\label{rem:read-rates}
In order to compute the Fourier coefficients of the scheme \eqref{eq:newScheme}, we apply $K-1$ times the operator $e^{it{\bf A}}$. Discretizing this exponential induces an $L^2$ error of order $K^{-s}$, yielding our rate $K^{-s+1}$ when summing the error terms.
As we measure the error either in $L^2$ or in a stronger $H^r$ norm, this explains why we need regularity at least $s>1$ to prove convergence rates.
\end{rem}

\subsection{Outline} The rest of the article proceeds as follows. In Section~\ref{sec:norms} we set the scene and introduce the spaces and norms we work with throughout the article, together with bilinear estimates which are used in the proof of the main theorem. Section~\ref{sec:explForm} contains the explicit formulas based on the Lax pair formulation. We derive our numerical schemes based on these formulas in Section~\ref{sec:newSchemes} and discuss their computational cost and accuracy, comparing them with existing schemes in the literature.  We give numerical experiments in Section~\ref{sec:num}, in the case of the Benjamin–Ono equation. After defining and establishing several tools crucial for the analysis in Section~\ref{sec:ops}, we prove in Section~\ref{subsec:pf} the spectral convergence results announced in Theorems~\ref{thm:main} and~\ref{thm:CS}.

\subsection*{Acknowledgements}
The authors would like to deeply thank Patrick Gérard for stimulating discussions and constructive feedback. We also thank Rana Badreddine for helpful remarks, and for her PhD defence where this project was started. Y.A.B also thanks Louise Gassot for fruitful discussions on the Benjamin–Ono equation.
The work of Y.A.B. is funded by the National Science Foundation through the award DMS-2401858 and M.D. acknowledges funding by the Deutsche Forschungsgemeinschaft (DFG, German Research Foundation) - Project number 442047500 through
the Collaborative Research Center “Sparsity and Singular Structures” (SFB 1481).

\section{Norms, spaces and Fourier transforms}\label{sec:norms}

Crucial for the analysis, and a common point of our three equations, is the space in which we study them. We define the Hardy space of functions whose Fourier transform is supported in $\mathbb N_0$ by
\begin{equation}
\label{eq:Hardy}
L^2_{+} = \{f\in L^2(\mathbb{T}) : \widehat{f}(k) = 0 \ \text{for} \ k<0 \},
\end{equation}
where the $L^2$ inner product and the Fourier coefficients are respectively defined as
\[
\langle f,g\rangle_{L^2}^2=\frac{1}{2\pi}\int_{-\pi}^\pi f(x)\overline{g(x)}dx\quad\text{and}\quad \widehat f(k)=\frac{1}{2\pi}\int_{-\pi}^\pi f(x)e^{-ikx}dx.
\]
For concision, we use the shorthand notation $\|f\|=\|f\|_{L^2}$ and $\langle f,g\rangle=\langle f,g\rangle_{L^2}$. With these definitions, Fourier inversion, Parseval identity and the product-convolution identity read as follows:
\[
f(x)=\sum_{k\in\mathbb Z}\widehat f(k) e^{ikx},\quad \|f\|_{L^2}^2=\sum_{k\in \mathbb Z} |\widehat f(k)|^2\quad\text{and}\quad \widehat{fg}=\widehat f*\widehat g.
\]

By identifying $\mathbb T$ with the unit circle in $\mathbb C$, the space $L^2_+$ can equivalently be characterized as the traces of holomorphic functions $f$ on the unit disk
\[
\mathbb{D} = \{z\in \mathbb{C} : |z| <1 \},
\]
satisfying
\[
\sup_{r<1}\,\frac{1}{2\pi}\int_{-\pi}^{\pi} |f(re^{ix})|^2\,dx < + \infty.
\]
The explicit formulas in the literature use this characterization, see equations \eqref{eq:explicit-BO-disc}, \eqref{eq:explicit-CS-disc} and \eqref{eq:explicit-S-disc}.
We point out that the previously mentioned Riesz–Szeg\H o operator \eqref{eq:Pi} is the orthogonal projector from $L^2$ to~$L^2_{+}$.

For $r>0$, we also introduce the Sobolev space $H^r=\{f\in L^2 : \|f\|_{H^r}<\infty\}$ with
\[
\|f\|_{H^r}^2=\|(I+D^2)^{r/2}f\big\|^2=\sum_{k\in \mathbb Z}(1+k^2)^{r}|\widehat f(k)|^2,
\]
and the Hardy–Sobolev space 
\begin{equation}
\label{eq:HardySobolev}
H^r_{+} = H^r\cap L^2_{+}.
\end{equation}

We immediately see that, for $r'<r$ and $f\in H^r$, $\|f\|_{H^{r'}}\leq \|f\|_{H^r}$. Moreover, the following bilinear estimate holds:
\begin{lem}
\label{lem:bilinear}
Let $s> 1/2$ and $0\le\sigma\leq s$. Then, there exists a constant $C_1 > 0$ such that for all $f\in H^s$ and $g\in H^\sigma$,
\[
\|fg\|_{H^\sigma}\leq C_1\|f\|_{H^s}\|g\|_{H^\sigma}.
\]
\end{lem}
The proof of the above lemma is quite standard, nevertheless we recall it in Section~\ref{sec:appendix} for completeness and traceability of the constants.

\section{Explicit formulas}\label{sec:explForm}

We now present the explicit formulas from \cite{G, Bad24,GG15}, written as inversion dynamical formulas defined inside the open unit disk, see equations \eqref{eq:explicit-BO-disc}, \eqref{eq:explicit-CS-disc} and  \eqref{eq:explicit-S-disc}. We derive from these formulas
a characterization of the Fourier coefficients $\widehat{u}(t,k)$ of the solution in terms of the initial data $u_0$ and the time $t$, see equations \eqref{eq:explicitForm}, \eqref{eq:explicitForm2}, \eqref{eq:explicitForm3}, and Remark~\ref{rem:Characterization}.
This later formulation is perfectly suited for approximating numerically, via a spectral discretization, as will be seen in Section~\ref{sec:newSchemes}.

Recalling the definition of the Riesz-Szeg\H o projector $\Pi:L^2 \rightarrow L^2_{+}$ from \eqref{eq:Pi} and \eqref{eq:Hardy}, we introduce another crucial operator, $S^*:L^2_+ \rightarrow L^2_{+}$,
which removes the zero-th Fourier coefficient and shifts all positive frequencies by one
\[
S^*f = \Pi(e^{-ix}f), \quad\text{i.e.}\quad \widehat{S^*f}(k)=\mathbbm{1}_{k\geq 0}\,\widehat{f}(k+1),\quad f \in L^2_{+}.
\]

We are now ready to state the explicit formulas.

\noindent {\bf Benjamin–Ono.} 
For \eqref{eq:BO}, it was discovered by Gérard \cite[Theorem 4]{G} that
\begin{equation}
\label{eq:explicit-BO-disc}
\Pi u(t, z)=\left\langle\left(I-z e^{i t} e^{2 i t L_{u_0}^{\text{BO}}} S^*\right)^{-1} \Pi u_0, 1\right\rangle, \quad \forall z \in \mathbb{D},
\end{equation}
where the Lax operator $L_{u_0}^{\text{BO}}$ is the semi-bounded self-adjoint operator defined on $H^1_{+}$ by
\[
 L_{u_0}^{\text{BO}}f = Df - \Pi(u_0f).
\]
By expanding formula \eqref{eq:explicit-BO-disc} into a Neumann series in $z=re^{ix}$
and letting $r$ tend to $1$,
we identify the Fourier coefficients of the solution
\begin{align}\label{eq:explicitForm}
\widehat{u}(t,k) =\left\langle (e^{it}e^{2itL_{u_0}^{\text{BO}}}S^{*})^k \Pi u_0, 1 \right\rangle, \quad k \ge 0.
\end{align}
We note that in the case $k<0$ we simply have $\widehat{u}(t,k) = \overline{\widehat{u}(t,-k)}$ since $u$ is real-valued.

We now comment on other explicit formulas existing in the literature.
A precursor to the inversion formula \eqref{eq:explicit-BO-disc} is the work of Gérard--Kappeler \cite[Lemma 4.1]{GK21} which considers finite gap initial conditions.
An explicit formula on the real line $\mathbb{R}$ is obtained by Gérard \cite[Theorem~6]{G}, and extended by the second author \cite{Chen} to less regular initial data $u_{0} \in L^{2}(\mathbb{R})$. In the case of rational initial data, an explicit formula on the real line is given in \cite{BGGM24-1}, expressed as a ratio of determinants.
A generalization of Gérard's explicit formula \cite[Theorem~6]{G} to the full hierarchy of \eqref{eq:BO} is presented in Killip--Laurens--Vi\c san~\cite{KLV24}.  
\bigskip

\noindent{\bf Calogero–Sutherland DNLS.}
Badreddine's explicit formula for  \eqref{eq:CS} in the focusing \cite[Proposition 2.6]{Bad24} and defocusing \cite[Theorem 1.7]{Bad24} case is given by
\begin{equation}
\label{eq:explicit-CS-disc}
u(t, z)=\left\langle\left(I-z e^{-i t} e^{-2 i t L_{u_0}^{\text{CS}}} S^*\right)^{-1} u_0, 1\right\rangle, \quad \forall z \in \mathbb{D},
\end{equation}
where the Lax operator $L_{u_0}^{\text{CS}}$ is the semi-bounded self-adjoint operator of domain $H^1_{+}$ given by
\[
L_{u_0}^\text{CS} f = Df \mp u_0 \Pi(\overline{u_0}f),
\]
where the signs $-$ and $+$ correspond to the focusing case and the defocusing case, respectively.
By the same procedure as above, we infer from formula \eqref{eq:explicit-CS-disc} the following characterization 
\begin{align}\label{eq:explicitForm2}
\widehat{u}(t,k) =\left\langle (e^{-it}e^{-2itL_{u_0}^{\text{CS}}}S^{*})^k  u_0, 1 \right\rangle, \quad k \ge 0.
\end{align}
We recall that the initial data belongs to a space $H^s_{+}$, defined by \eqref{eq:HardySobolev}, hence $\widehat{u}(t,k) = 0$ for $k<0$.
We refer to Killip--Laurens--Vi\c san \cite{KLV23} for an explicit formula on the real line $\mathbb{R}$, and to Sun \cite{SunCS} for a matrix valued-extension.
\bigskip

\noindent{\bf Cubic Szeg\H o.}
The explicit formula found by Gérard and Grellier \cite[Theorem 1]{GG15} reads
\begin{equation}
\label{eq:explicit-S-disc}
u(t, z)=\left\langle\left(I-z e^{-i t H_{u_0}^2} e^{i t K_{u_0}^2} S^*\right)^{-1} e^{-i t H_{u_0}^2} u_0, 1\right\rangle, \quad \forall z \in \mathbb{D},
\end{equation}
where the self-adjoint operators $H_{u_0}$ and $K_{u_0}$ defined on 
$L^2_{+}$ are given by
\[
H_{u_0}(f) = \Pi(u_0 \bar{f })\quad \text{and} \quad K_{u_0}^2f=H_{u_0}^2f-\langle f, u_0\rangle u_0, \quad f \in L^2_{+}.
\]
Once again, we infer from the above the characterization in Fourier
\begin{align}\label{eq:explicitForm3}
\widehat{u}(t,k) =\left\langle (e^{-itH_{u_0}^2}e^{itK_{u_0}^2}S^{*})^k e^{-itH_{u_0}^2}u_0, 1 \right\rangle, \quad k \ge 0.
\end{align}
As for the \eqref{eq:CS} equation, we have $\widehat{u}(t,k) = 0$ for $k<0$.

An explicit formula was also derived for matrix valued extensions of \eqref{eq:S} in Sun \cite{SunS}. On $\mathbb{R}$, explicit formulas were found by Pocovnicu \cite{P11} and Gérard--Pushnitski \cite{GP24}.

\begin{rem}\label{rem:Characterization}
The characterization in Fourier \eqref{eq:explicitForm} already appeared in \cite[Remark 5]{G} for \eqref{eq:BO}, and allowed to extend the explicit formula down to more singular initial data $u_0 \in H^s$, with $s>-\frac{1}{2}$.  
\end{rem}

\section{New schemes based on the explicit formulas}\label{sec:newSchemes}

\subsection{Construction of the schemes}

In this section we present the three numerical schemes for the \eqref{eq:BO}, \eqref{eq:CS} and \eqref{eq:S} equations, derived from the explicit formulas \eqref{eq:explicitForm}, \eqref{eq:explicitForm2} and \eqref{eq:explicitForm3} respectively. We construct schemes of the general form~\eqref{eq:newScheme}, by restricting all operators to the $K$ frequencies $(0, \dots, K-1)$.

We discretize in $\mathbb C^{K\times K}$ the shift operator, the derivative and the convolution with $u_0$ as 
\[
{\bf S}^*=(\mathbbm 1_{k+1=\ell})_{0\leq k,\ell < K}\quad 
{\bf D} = (k\mathbbm 1_{k=\ell})_{0\leq k,\ell < K}\quad \text{and}\quad
{\bf T}_{u_0} = (\widehat{u_0}(k-\ell))_{0\leq k,\ell < K}.
\]
Observe that for \eqref{eq:BO}, ${\bf T}_{u_0}$ is hermitian because $u_0$ is real-valued, while for \eqref{eq:CS} it is lower triangular since $u_0\in L^2_+$.

Introducing the discretization ${\bf D} - {\bf T}_{u_0}$ of the Lax operator $L_{u_0}^{\text{BO}}$, the scheme for \eqref{eq:BO} is obtained by taking 
\begin{equation}
\label{AM-BO}
{\bf A} = {\bf I} +2 {\bf D} -2 {\bf T}_{u_0} \quad \text{and} \quad {\bf M} = 0 
\end{equation}
in equation \eqref{eq:newScheme}.

For \eqref{eq:CS}, we let ${\bf T}_{u_0}^*$ denote the conjugate transpose of ${\bf T}_{u_0}$, which  corresponds to a convolution with $\overline{u_0}$. We similarly recover the scheme by taking
\begin{equation}
\label{AM-CS}
{\bf A} = {-\bf I} - 2 {\bf D} \pm 2{\bf T}_{u_0}{\bf T}_{u_0}^{*} \quad \text{and} \quad {\bf M} = 0,
\end{equation}
where the signs $+$ and $-$ correspond to the focusing case and the defocusing case, respectively.

Finally, for \eqref{eq:S}, to take into account the conjugation of the argument of $H_{u_0}$, we modify the convolution matrix as follows
\[
{\bf H}_{u_0} = (\widehat{u_0}(k+\ell))_{0\leq k,\ell < K}.
\]
We then define the scheme through the choices
\begin{equation}
\label{AM-S}
{\bf A} = {\bf H}_{u_0}{\bf H}_{u_0}^*- {\bf u}_0{\bf u}_0^*
\quad \text{and}\quad {\bf M} = {\bf H}_{u_0}{\bf H}_{u_0}^*,
\end{equation}
which are truncations of the operators $K_{u_0}^2$ and $H_{u_0}^2$, respectively. 

\begin{rem}[Different formulations of the scheme]
The above schemes are written in the form to be implemented. 
We can write the schemes --- as is done in Section~\ref{sec:truncated-ops} for \eqref{eq:BO} --- in a more theoretical fashion using only the operator $\Pi_K$, which is better suited for analyzing their convergence.
\end{rem}

\subsection{Comparison with schemes in the literature}

\label{sec:comparison}
The above schemes are computed {\it entirely in Fourier space}. To understand why this yields efficient algorithms, we need to consider  their computational cost together with their precision. We note that since no numerical schemes or convergence results have been proposed for \eqref{eq:CS} and \eqref{eq:S} prior to this work, we will focus on comparing our method to existing fully discrete schemes and their convergence results for \eqref{eq:BO}, which has been more extensively studied numerically.

For our scheme \eqref{eq:newScheme} the accuracy $\epsilon = \|u(t) - u_{K}(t) \|_{H^r}$ and computational cost $\mathcal{C}$ are of order \footnote[3]{This rate $\epsilon$ holds for \eqref{eq:BO} and \eqref{eq:CS}, while the cost $\mathcal{C}$ also holds for \eqref{eq:S}.}
\[
\epsilon \sim TK^{-s+1+r} \quad \text{and} \quad \mathcal{C} \sim K^{3},
\]
where $T = t$ is the final time and $K$ the number of frequencies in the discretization.
We note that the leading cost comes from computing the matrix exponentials in equation \eqref{eq:newScheme}. Indeed, since the matrices are self-adjoint, they can be diagonalized as $P\Lambda P^T$, allowing us to compute $Pe^{it\Lambda}P^T$ in $O(K^3)$ operations. Moreover, as the eigenvalues $\lambda_j$ of the diagonal matrix $\Lambda$ are real-valued, we have $|e^{it\lambda_j}| = 1$ for all $t$, ensuring the stability of the method over time.
Once the matrix exponentials are computed, we can evaluate all the vectors $\left(e^{-it{\bf M}}e^{it{\bf A}}{\bf S}^*\right)^k e^{-it{\bf M}}\,{\bf u}_0$ in $\mathcal O(K)$ matrix-vector multiplications, with a computational cost in $\mathcal O(K^3)$ once again.
It follows from the above that the cost required to reach an accuracy $\epsilon$ is of order
\begin{equation}
\label{eq:cost-eps}
\mathcal{C} \sim \left(T\epsilon^{-1}\right)^{\frac{3}{s-1-r}}.
\end{equation}

In comparison, prior convergence results of fully discrete schemes $u^{n}_{{K}}$ for \eqref{eq:BO} analyze a finite difference time-approximation coupled with either a spectral Galerkin method \cite{PD01, DM09, G16} or a pseudo-spectral finite difference method \cite{T98} in space. The former schemes have a cost at least  $\mathcal{C} \sim \frac{T}{\tau}K^2$ while the later costs $\mathcal{C} \sim \frac{T}{\tau}K\log K$, with $\tau$ the time-step and $T$ the final time.
Given that the cost $\mathcal{C}$ of our scheme is independent of $T$, and that the constant in the error term $\epsilon$ grows linearly with $T$ --- in contrast to the exponential growth $e^T$ of prior schemes --- our method outperforms previous approaches for long-time computations $T \gg 1$.
We next discuss how our new scheme improves prior results even over short time $T = O(1)$.

We begin by discussing all the fully-discrete results \cite{T98, PD01, G16}, except that of Deng--Ma \cite{DM09}, which we treat separately. All these former works require a CFL condition of the form $\tau \lesssim K^{-2}$, resulting in a computational cost of order $O(TK^4)$ or $O(TK^3\log(K))$ respectively. In terms of accuracy, defined as $\epsilon = \|u(n\tau) - u^{n}_{K}\|_{H^r}$, and for a fixed regularity $s$ of the initial data in the regimes considered by the authors, our scheme achieves a better convergence rate than these earlier results, thereby providing an improved method. 

In Deng--Ma \cite{DM09}, the authors establish a convergence rate of order $e^{T}(\tau^2 + K^{-s +r})$, with $r = 1/2$ and $s\ge 17/2$. In particular, for times $T = O(1)$, they achieve accuracy $\epsilon$ if $\tau\lesssim \sqrt\epsilon$ and $K\gtrsim \epsilon^{-\frac{1}{s-r}}$, leading to a computational cost bounded below by
\[
\mathcal{C} \sim \frac{T}{\tau}K^2\gtrsim \epsilon^{-\frac{1}{2}-\frac{2}{s-r}},
\]
which exceeds our cost \eqref{eq:cost-eps}, given their regularity requirement \footnote[4]{In fact, our algorithm is less costly as soon as $s >4+r=\frac92$.}. 

Therefore, over both short times $T= O(1)$ and long times $T\gg1$, our fully-discrete algorithm improves prior results. In the following section, we witness the strength of our new numerical method for both smooth and less regular initial data, over short and long time intervals.
\newline

\subsection{Numerical simulations in the case of the Benjamin–Ono equation}\label{sec:num}

Although there is a vast literature on different numerical schemes for the \eqref{eq:BO} equation, we choose to compare ours with the scheme consisting of coupling a Fourier pseudo-spectral method with a standard explicit 4-stage Runge-Kutta (RK4) time-stepping method. Despite the fact that, up to our knowledge, no convergence results exists for this scheme, it remains a very popular method to obtain an efficient high order approximation, see for example \cite{BX11}. To ensure stability of the method we impose a CFL condition of the form $\tau\le C h^2$, where $h=\frac{2\pi}K$ is the spatial mesh size. In the following numerical simulations we take $C=\frac{1}{4}$. 

This pseudo-spectral method is efficient as it has a computational cost in $\frac{T}{\tau}K\log(K)$ when computing up until the final time $T$. Given the quadratic CFL condition its cost is of order $TK^3\log K$. This is to be compared with the cost of the new scheme \eqref{eq:newScheme} which is of order $K^3$.

\subsubsection{An exact smooth solution}

Numerical results are first presented for the $2\pi$-periodic travelling wave solutions  
\begin{equation}
\label{eq:exactSmooth}
u^{*}(t,x) = \frac{1}{c-\sqrt{c^2-1}\cos(x-c t)}, \quad c >1.
\end{equation}
These travelling waves, obtained by Benjamin \cite{BO}, were proved by Amick-Toland \cite{AT91} to be unique. We note that when $c>1$, the solution $u^{*}$ is real and forms a single solitary wave. In the following we take either $c = \frac{15}{4\pi}$, in agreement with the example of \cite{T98}, or $c = \frac{15}{\pi}$ which corresponds to a tighter peak.

\begin{figure}[ht]
\centering
\includegraphics[height=6cm]{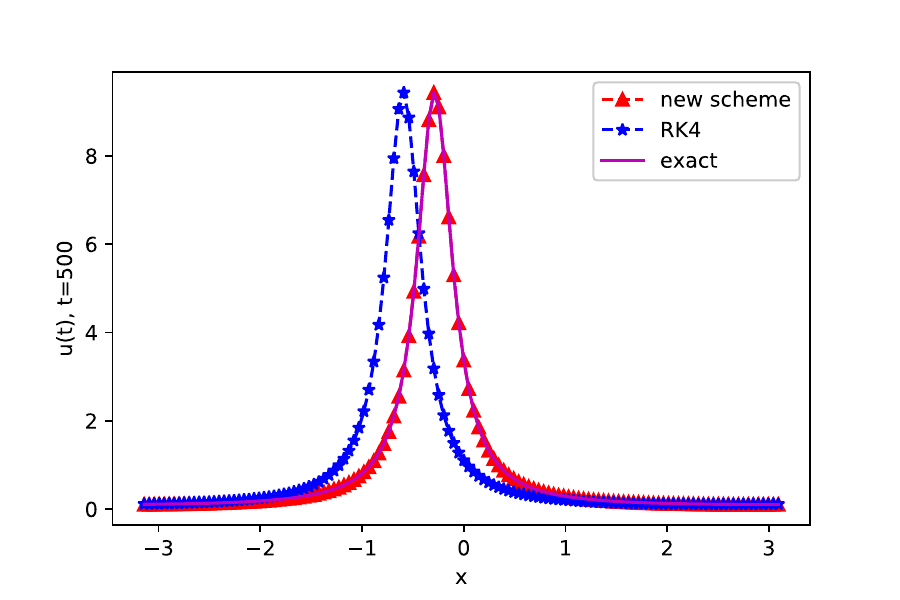}
\caption{Plot of the solution \eqref{eq:exactSmooth} in purple (with $c = \frac{15}{\pi}$), the new scheme \eqref{eq:newScheme} in red, and the pseudo-spectral RK4 method in blue. We choose $t=500$ and $K=128$. The initial profile is translated at constant speed $c$, thus it has periodically returned near the origin $ct/2\pi\approx 380$ times between $0$ and $t$.
}
\label{long-profil}
\end{figure}

We show in numerical simulations how the new scheme \eqref{eq:newScheme} clearly outperforms previous schemes in the literature, both in the case of {\it short} (Figure~\ref{short-CPU}) and {\it long} (Figures~\ref{long-CPU},~\ref{long-profil}) times, and compare it with the pseudo-spectral RK4 scheme. In Figure~\ref{short-CPU} and~\ref{long-CPU} we chose as final times $T=1$ and $T=100$ respectively, and compute the CPU-time versus $L^2$-error of the scheme for varying time and space step sizes. We see that the new scheme is far more {\it precise}. This is thanks to the fact that it is {\it exact in time}, with spectral accuracy in space, hence the error decreases faster than any polynomial. In contrast, for smooth solutions, the error of any {\it fully discrete pseudo-spectral} scheme existing in the literature is {\it dominated by the time discretization error} of order $\tau^m$, for some fixed $m\in \mathbb{N}$, which hence induces a larger error.
Our schemes also perform very well for large times since its CPU cost is independent of the final time and the error constant only grows {\it linearly} in time, see Theorem~\ref{thm:main} and Figure~\ref{fig:lin-err}.
We refer to Figure~\ref{long-profil} where the exact periodic solution and numerical approximations are plotted at time $t=500$, we see that only the new scheme gives a reliable approximation. The CPU times needed to compute these schemes is $215$\,s for the RK4 method versus $6.28\times10^{-3}$\,s for the new scheme.

\subsubsection{Less regular solutions}\label{sec:num-rough}
In this section we test our scheme in the case of a randomized initial data, constructed as follows:
we take a vector $(\widehat{U}_{k})_{k\in \mathbb Z}$ sampled from a gaussian distribution while ensuring that the initial datum is real-valued. Namely, we draw independently $\widehat{U}_0 \sim \mathcal{N}(0,1)$ and  $\widehat{U}_k \sim \mathcal{N}(0,1/2)+ i\mathcal{N}(0,1/2)$ for $0<k < K_{\text{ref}}$, and take $\widehat{U}_{k} = \overline{\widehat{U}_{-k}}$ for $-K_{\text{ref}}< k < 0$. Letting $s$ denote the regularity parameter, we consider the truncated initial data
\begin{equation}
\label{eq:randomID}
u_{0}(x) = \sum_{-K_{\text ref}< k< K_{\text ref}} (1+|k|)^{-(s+\theta)}\widehat{U}_{k} e^{ikx}, \quad x\in \mathbb T,
\end{equation}
and rescale it in $l^2$-norm by applying $u_{0} \leftarrow u_{0}/\| u_{0}\|_{l^2}$. In the above expression, we take $\theta > 1/2$ in order to have $\lim_{K_{\text{ref}}\to \infty}u_{0} \in H^s(\mathbb{T})$. The reference solution is computed using our scheme, with $K_{\text{ref}} = 1024$ Fourier modes,
and the numerical simulations are performed with $s=2$, mimicking in a discrete setting the case of a rougher solution belonging to $H^2$.

\begin{figure}[ht]
    \centering
    \begin{subfigure}[t]{0.5\textwidth}
        \centering
        \includegraphics[height=6cm]{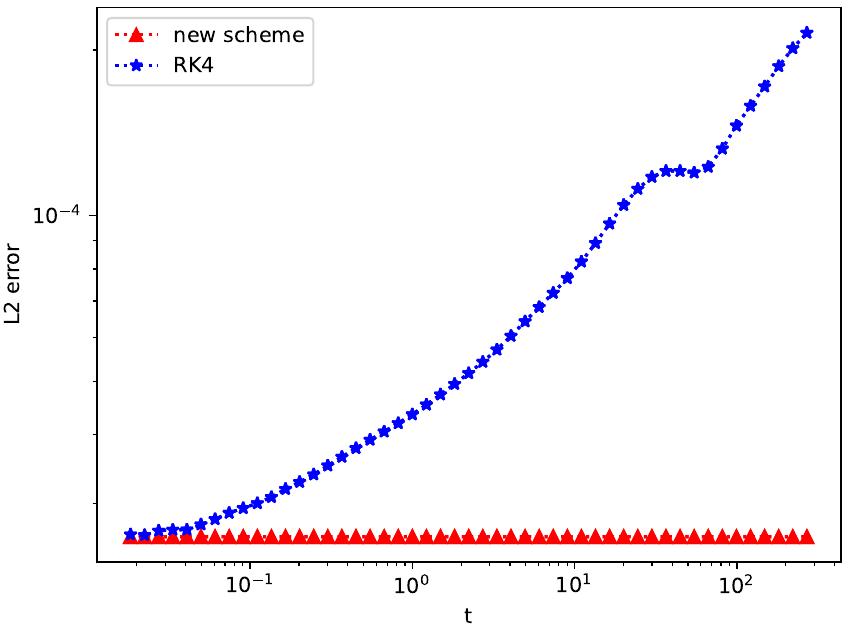}
        \caption{}\label{fig:exp-growth-RK4}
    \end{subfigure}%
    ~ 
    \begin{subfigure}[t]{0.5\textwidth}
        \centering
        \includegraphics[height=6cm]{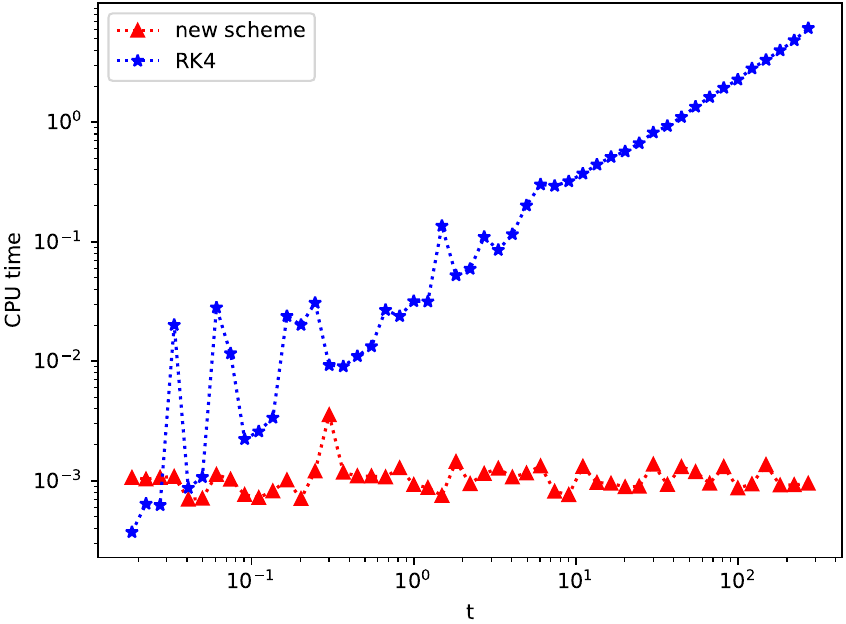}
        \caption{}\label{fig:errCst-CPU}
    \end{subfigure}
    \caption{We let $K=64$. Left: error in $L^2$ norm over time ($t \in [10^{-2}, 10^2]$) for the randomized initial data \eqref{eq:randomID} with $s = 2$. For these values of $t$, the error of our scheme (red) is dominated by the truncation error $K^{-s}$ of the initial datum, which is constant in time.
     Right: corresponding computational cost of the schemes over time $t$. The cost grows linearly for RK4 (blue) while remaining constant for our scheme (red).}
\end{figure}

First, in Figure~\ref{fig:exp-growth-RK4}, we plot the $L^2$ error up to times of order $10^{2}$. We observe that the error remains constant for the new scheme on this time interval, while growing faster than linearly for the RK4 scheme.
Note that the expected linear error growth over time for our scheme (as in Figure~\ref{fig:lin-err}) appears when computing up to much larger times, of order $10^{7}$. In Figure~\ref{fig:errCst-CPU}, we plot the corresponding CPU cost over the same time interval. As expected, our scheme maintains a constant cost over time, in contrast to the linear cost growth observed with the RK4 method.

\begin{figure}[ht]
    \centering
    \begin{subfigure}[t]{0.5\textwidth}
        \centering
        \includegraphics[height=6cm]{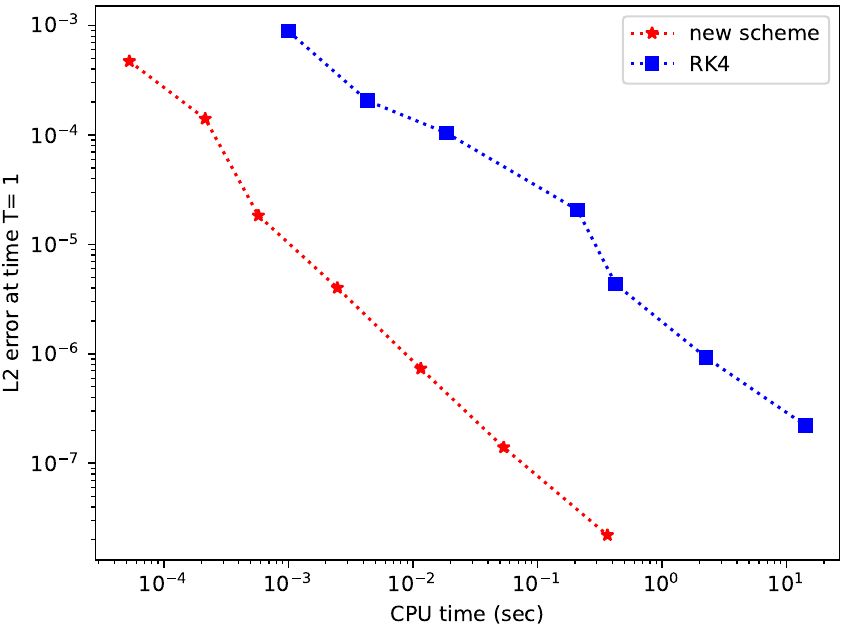}
    \end{subfigure}%
    ~ 
    \begin{subfigure}[t]{0.5\textwidth}
        \centering
        \includegraphics[height=6cm]{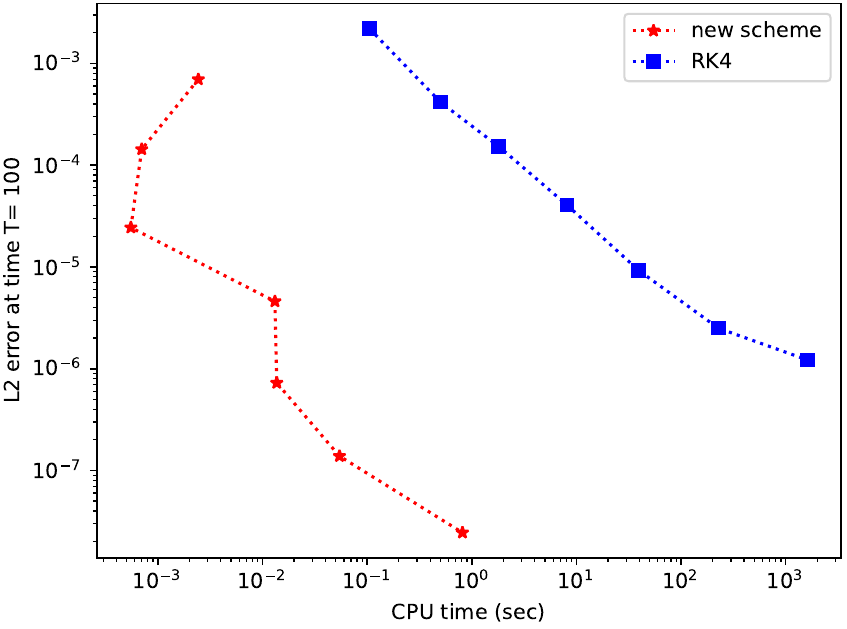}
    \end{subfigure}
    \caption{Convergence plot for the \eqref{eq:BO} equation in $L^2$ against the computational cost up to time $T=1$ (left)  and $T=100$  (right) for the randomized initial data \eqref{eq:randomID} with $s = 2$. We choose the number of Fourier modes $K$ to be powers of two ranging from 8 to 512.}\label{fig:H2-BO}
\end{figure}

In Figure~\ref{fig:H2-BO}, we plot the $L^2$ error against CPU cost. The simulations highlight the advantage of employing our new scheme both in the case of short and long times.
We observe that our scheme continues to outperform RK4 even for lower regularity data, such as in the case $s = 1$.
The same favorable behavior is observed numerically when our scheme \eqref{eq:newScheme} is used to approximate \eqref{eq:CS} and~\eqref{eq:S}.

Having motivated in numerical simulations the advantages of the new scheme \eqref{eq:newScheme}, we now prepare the ground for proving its convergence and introduce in the following section some notation and definitions of operators used in the proof.

\section{Proving convergence}

We recall that in this section we consider the \eqref{eq:BO} equation, whose solution and initial data are real-valued functions.

\subsection{Prerequisites for the proof}\label{sec:ops}
\subsubsection{The Lax Pair}\label{sec:LaxPairs}
Given $u \in H^2$, we can define the following Toeplitz operator on $L_{+}^2$,
\begin{align*}
\forall f \in L_{+}^2, \quad T_u f=\Pi(u f).
\end{align*}
With the above notation, we recall the Lax operator $L_{u}$ for \eqref{eq:BO} already introduced in Section~\ref{sec:explForm},
\begin{align*}
L_u=D - T_u.
\end{align*}
In the proof, we will use the second Lax operator $B_u$, which is a bounded skew-adjoint operator defined by
\begin{align}\label{eq:B_u}
B_u =i\left(T_{|D| u}-T_u^2\right),
\end{align}
as well as the following two propositions whose proofs are given in \cite{G}.
\begin{prop}{\cite[Corollary 3]{G}}
\label{lemma 4.2}
Let $u(t)$ be the solution of \eqref{eq:BO} with initial data $u_0 \in H^2$. Denote by $U(t)$ the operator-valued solution of the linear ODE
\[
U^{\prime}(t)=B_{u(t)} U(t), \quad U(0)=I.
\]
Then  for every $t \in \mathbb{R}$, $U(t)$ is unitary on $L_{+}^2$, and
\[
L_{u(t)}=U(t) L_{u_0} U(t)^*.
\]
\end{prop}
\noindent During the derivation of the explicit formula in \cite[Section 2]{G},
at the bottom of page 597, Gérard discovered the following identity.
\begin{prop}
\label{lemma 4.3}
Under the condition of Proposition~\ref{lemma 4.2}, we have
\[ 
U(t)^* S^* U(t)=e^{i t\left(L_{u_0}+I\right)^2} S^* e^{-i t L_{u_0}^2}.
\]
\end{prop}
\begin{rem}
The operator $U(t)$ can be shown to be unitary on $L^2_{+}$ for $u(t) \in H^s$, $s>3/2$. Indeed, the regularity requirement stems from equation \eqref{eq:B_u} where a standard bilinear estimate requires $|D|u \in L^\infty$. Hence, the above two lemmas can be stated for $u(t)$ in these weaker spaces. Nevertheless, to be consistent with prior works we keep the stronger hypothesis $u(t) \in H^2$, as this does not change the steps in our proof which follow by density for $s<2$.
\end{rem}

\subsubsection{Truncated Lax operator} 
\label{sec:truncated-ops}
Recalling the definition of $\Pi_K$ from Section~\ref{sec:results}, we define the operator  $L_{u_0,K}$ by
\[
L_{u_0,K}f=Df-\Pi_K(u_0\Pi_Kf), \quad f \in L^{2}_{+},
\]
and let
\[
A_K=I+2L_{u_0,K}\quad \text{and}\quad A=I+2L_{u_0}.
\]
Note that for any function $f$ in
\[
L^2_K=\Big\{f\in L^2,\;{\rm supp}(\widehat f)\subset \{0,\dots,K-1\}\Big\}
\]
and any $0\leq k<K$, it holds
\[
\widehat{A_Kf}(k)=\widehat f(k)+2k\widehat f(k)-2\sum_{\ell=0}^{K-1}\widehat u_0(k-\ell)\widehat f(\ell) = \sum_{\ell=0}^{K-1}{\bf A}_{k,\ell}\widehat f(\ell),
\]
thus the restriction of $A_K$ to $L^2_K$ corresponds to the matrix ${\bf A}$ from \eqref{AM-BO} in the Fourier basis.
According to equations \eqref{eq:newScheme} and \eqref{eq:explicitForm}, it follows that for $k\in\{0,\dots,K-1\}$,
\[
\widehat{u_K}(t,k)=\langle(e^{itA_K}S^*)^k\Pi_K u_0,1\rangle
\quad \text{and}\quad
\widehat u(t,k)=\langle(e^{itA}S^*)^k\Pi u_0,1\rangle.
\]
\begin{rem}
In the computation of $u_K$, we only apply $e^{itA_K}$ to functions $f\in L^2_K$,
for which $L_{u_0,K}f=Df-\Pi_K(u_0f)$. However, we need a second $\Pi_K$ in the definition of the Toeplitz term in order to make $L_{u_0,K}$ self-adjoint on $L^2_+$.
In contrast, for the first term, while $D$ and $\Pi_KD=D\Pi_K$ coincide on $L^2_{K}$, they are both self-adjoint on $L^2_{+}$. In view of the convergence analysis, and namely Lemma~\ref{lem:diff-exp}, it is important not to apply any truncation $\Pi_K$ on $D$ in order for the derivatives in $A_K-A = 2(L_{u_0,K}-L_{u_0})$ to cancel out.
\end{rem}

With the above definitions, the operators $e^{itA}$ and  $e^{itA_K}$ preserve the $L^2$ norm.
\begin{lem}
\label{lem:self-adjoint}
For any $f\in L^2_+$, $\|e^{itA_K}f\|_{L^2}=\|f\|_{L^2}$ and $\|e^{itA}f\|_{L^2}=\|f\|_{L^2}$.
\end{lem}
\begin{proof}
For $f,g\in L^2_+$, as $u_0$ is real-valued,
\[
\langle\Pi_K(u_0\Pi_Kf),g\rangle=\langle u_0\Pi_Kf,\Pi_Kg\rangle=\langle\Pi_Kf,u_0\Pi_Kg\rangle=\langle f,\Pi_K(u_0\Pi_Kg)\rangle
\]
and
\[
\langle Df,g\rangle= \sum_{k\geq 0} k\widehat f(k)\overline{\widehat g(k)}=\langle f,Dg\rangle.
\]
Therefore $L_{u_0,K}$ is self-adjoint, and so is $A_K$. As a consequence,
\[
\frac{d}{dt}\|e^{itA_K}f\|_{L^2}^2=\langle(iA_K-iA_K^*)e^{itA_K}f,e^{itA_K}f\rangle=0,
\]
and the first equality follows by integrating the last equation between $0$ and $t$. The second one is obtained in a similar fashion, by replacing $\Pi_K$, $L_{u_0,K}$ and $A_K$ by $\Pi$, $L_{u_0}$ and~$A$, respectively.
\end{proof}

\subsubsection{Equivalent norms}\label{sec:equivN}

In the next lemma, we assume that $t$ and $u_0$ are fixed, and that $s\ge 1$.
By Remark~\ref{remC1}, the constant~$C_1$ from Lemma~\ref{lem:bilinear} is bounded by $2^{s+2}$. In the sequel, we will define constants $C_i$ for $i>1$, that are allowed to depend on $s$, $\|u_0\|_{H^s}$ and $\|u(t)\|_{H^s}$,
but not explicitely on the final time~$t$.
The following equivalence of norms holds.
\begin{lem}
\label{lem:LuplusC2}
For $u\in \{u_0,u(t)\}$ and $f\in H^{m}_+$ with $m=\lfloor s\rfloor+1$, it holds
\[
C_3^{-m}\|f\|_{H^{m}}\leq \|(L_u+C_2 I)^mf\| \leq C_3^m \|f\|_{H^{m}},
\]
where $C_2=C_1\max(\|u_0\|_{H^s},\|u(t)\|_{H^s})+1$ and $C_3=2C_2$.
\end{lem}
\begin{proof}
For any $n\in\{0,\dots,m\}$, denote
\[
f_n=(L_u+C_2 I)^{m-n}f.
\]
For $n<m$ and $g\in H^{n+1}$, Lemma~\ref{lem:bilinear} shows that $\|ug\|_{H^{n}}\leq (C_2-1)\|g\|_{H^n}$, and hence
\[
\|(L_u+C_2 I)g\|_{H^{n}}\leq \|Dg\|_{H^{n}}+\|ug\|_{H^{n}}+C_2\|g\|_{H^{n}}\leq 2C_2\|g\|_{H^{n+1}}.
\]
In particular,
\[
\|f_n\|_{H^n}\leq C_3\|f_{n+1}\|_{H^{n+1}},
\]
which proves the upper bound $\|f_0\|\leq C_3^m \|f_m\|_{H^{m}}$ by induction.
\newline

For the lower bound, we first note that for any $g\in H^{n+1}_+$,
\[
\langle(L_u+(C_2-1)I)g,g\rangle_{H^n} = \langle D g,g\rangle_{H^n} -\langle ug,g\rangle_{H^n}+(C_2-1)\|g\|_{H^n}^2\geq 0,
\]
hence
\[
\|(L_u+C_2I)g\|_{H^{n}}^2=\|(L_u+(C_2-1)I)g\|_{H^{n}}^2+2\langle(L_u+(C_2-1)I)g,g\rangle_{H^n} +\|g\|_{H^n}^2\geq \|g\|_{H^n}^2.
\]
From this, we obtain
\[
\|g\|_{H^{n+1}}\leq \|(D+I)g\|_{H^n}\leq \|(L_u+C_2I)g\|_{H^{n}}+\|ug\|_{H^{n}}+(C_2-1)\|g\|_{H^{n}}\leq 2C_2\|(L_u+C_2 I)g\|_{H^{n}},
\]
and we conclude again by induction on $\|f_n\|_{H^n}$.
\end{proof}

The above lemma will be used in the next section to bound terms of the form $\|e^{itA}f\|_{H^s}$, see Lemma~\ref{lem:Hs-bound}.
Interestingly, it also allows to explicitly control the $H^m$ norm of the solution at time $t$ by an explicit function of $\|u_0\|_{H^m}$, for any integer $m\leq s$, as discussed in Remark~\ref{cor:norm_of_u(t)} below.
\begin{rem}
\label{cor:norm_of_u(t)}
When $s$ is integer, the term $\|u(t)\|_{H^s}$ appearing in the final constant \eqref{eq:cstFinal} of Theorem~\ref{thm:main} is bounded {\it explicitly} and uniformly in time. Namely, given  any integer $1\leq m\leq s$ and $t\in \mathbb R$, we have
\[
\|u(t)\|_{H^m}\leq 2^{3^{m+1} m!}(1+\|u_0\|_{H^m})^{3^mm!}\|u_0\|_{H^m}.
\]
Indeed, by \cite[equation (1.15)]{KLV24}, if the solution $u$ is in $H^m$, the conservation laws
\[
\langle L_{u(t)}^k\Pi u(t),\Pi u(t)\rangle=\langle L_{u_0}^k\Pi u_0,\Pi u_0\rangle
\]
hold for $0\leq k\leq 2m$. Combining these with Lemma~\ref{lem:LuplusC2}, where $s$ is replaced by $m-1$, yields
\[
C_3^{-m}\|\Pi u(t)\|_{H^m}
\leq \|(L_{u(t)}+C_2 I)^m \Pi u(t)\|
=\|(L_{u_0}+C_2 I)^m \Pi u_0\|\leq C_3^m\|\Pi u_0\|_{H^m},
\]
and therefore $\|u(t)\|_{H^m}\leq 2C_3^{2m}\|u_0\|_{H^m}$, with a constant $C_3=2+2^{m+2}\max(\|u_0\|_{H^{m-1}},\|u(t)\|_{H^{m-1}})$.
We conclude by induction on $m\geq 1$, using bilinear estimates to treat the base case $m=1$, since Lemma~\ref{lem:LuplusC2} cannot be applied with $s=0$.
\end{rem}

\subsection{The proof of convergence}\label{subsec:pf}
In this section we prove Theorem~\ref{thm:main}. We summarize in the following sentences the sequence of steps needed to complete the proof, which differs very much from classical techniques to show convergence of schemes (by coupling a local error and stability bound). It requires a deep understanding of the Lax pairs, their commutation properties with the shift operator $S^{*}$ on the Hardy space $L^2_+$, and of the explicit form of the solution \eqref{eq:explicit-BO-disc}. Indeed, while the error committed by the projection $\Pi - \Pi_{K}$ is trivially of order $O(K^{-s})$, the error made by discretizing the Lax operator $L_{u_{0}}$, and hence the term $(e^{itA}S^{*})^k - (e^{itA_K}S^{*})^k$, is much harder to control. In order to buckle the proof we first bound, in Lemma~\ref{lem:diff-exp}, the error of approximation of the linear flow $e^{itA}$. The bound involves the $H^s$ norm of a function $u^k$ related to the solution $u$, which needs to be controlled. This is done in Lemma~\ref{lem:Hs-bound}, which is the most technical part of the proof and calls upon the second Lax operator $B_{u}$, the identities introduced in Section~\ref{sec:LaxPairs}, and the equivalence of norms in Section~\ref{sec:equivN}. 
The proof of the theorem then proceeds by induction on the Fourier coefficients, without Gronwall-type argument, and thereby allows to obtain a global bound with a linear dependence on the final time~$t$.

\begin{lem}
\label{lem:diff-exp}
For $f\in L^2_+$, $t\geq 0$ and $s>1/2$,
\[
\|e^{itA}f-e^{itA_K}f\|\leq 4C_1\|u_0\|_{H^s}\,t\,K^{-s} \sup_{t'\in[0,t]}\|e^{it'A}f\|_{H^s}.
\]
\end{lem}
\begin{proof}
We let
\[
F(t')=e^{i(t-t')A_K}e^{it'A}f,
\]
and observe that
\begin{align*}
\|e^{itA}f-e^{itA_K}f\|
&=\|F(t)-F(0)\|=\left\|\int_0^t \frac{dF}{dt'}dt'\right\|\\
&\leq \int_0^t\|e^{i(t-t')A_K}(A_K-A)e^{it'A}f\|dt'\\
&= \int_0^t\|(A_K-A)e^{it'A}f\|dt',
\end{align*}
where we used Lemma~\ref{lem:self-adjoint} in the last equality.
For $g=e^{it'A}f$, we have
\[
\frac{1}{2}(A_K-A)g=\Pi(u_0g)-\Pi_K(u_0\Pi_Kg)=(\Pi-\Pi_K)(u_0g)-\Pi_K(u_0(g-\Pi_Kg)),
\]
so we conclude with
\begin{align*}
\|(A_K-A)g\| 
&\leq 2\|(\Pi-\Pi_K)(u_0g)\|+2\|\Pi_K(u_0(g-\Pi_Kg))\| \\
&\leq 2\|u_0g\|_{H^s}K^{-s}+2C_1\|u_0\|_{H^s}\|g-\Pi_Kg\| \\
&\leq 4C_1\|u_0\|_{H^s}\|g\|_{H^s}K^{-s},
\end{align*}
with the constant $C_1$ from Lemma~\ref{lem:bilinear}.
\end{proof}

\begin{lem}
\label{lem:Hs-bound}
Given $t\geq 0$ and an integer $k\geq 0$, let $u^k=(e^{itA}S^*)^k \Pi u_0$. Then for any $s>1$ and $\tilde t\in [0,t]$,
\[
\|e^{-i\tilde t A}u^k\|_{H^s} \leq C_3^{4s}\|u_0\|_{H^s},
\]
where $C_3$ is the constant defined in Lemma~\ref{lem:LuplusC2}.
\end{lem}
\begin{proof}
We first assume that $u_0\in H^2$, in order to ensure that $B_{u(t)}$ and $U(t)$ are well-defined.
By definition of $A$ and Proposition~\ref{lemma 4.3}, we have
\[
e^{itA}S^*=e^{it+2itL_{u_0}}S^*= e^{-i t L_{u_0}^2} e^{i t\left(L_{u_0}+I\right)^2} S^* = e^{-i t L_{u_0}^2} U(t)^*S^*U(t) e^{i t L_{u_0}^2}.
\]
By induction, we thereby obtain
\[
(e^{itA}S^*)^k=e^{-i t L_{u_0}^2} U(t)^* (S^*)^k U(t) e^{i t L_{u_0}^2},
\]
so $e^{-i\tilde tA}u^k=P(S^*)^kQ\Pi u_0$, with
\[
P=e^{-i\tilde t A}e^{-i t L_{u_0}^2} U(t)^*\quad\text{and}\quad Q=U(t) e^{i t L_{u_0}^2}.
\]
As $A$ and $L_{u_0}$ are self-adjoint, and $U(t)$ is unitary, for any $f\in L^2_+$,
\[
\|Pf\|=\|Qf\|=\|f\|.
\]
Moreover, 
from Proposition~\ref{lemma 4.2} we have
\[
(L_{u_0}+C_2I)^m U^{*} = U^{*}(L_{u(t)}+C_2I)^m,
\]
and hence, as $L_{u_0}$ commutes with $e^{-i\tilde t A}$ and $e^{-i t L_{u_0}^2}$,
\[
(L_{u_0}+C_2I)^mP=e^{-i\tilde t A}e^{-i t L_{u_0}^2}(L_{u_0}+C_2I)^m U^{*}=P(L_{u(t)}+C_2I)^m.
\]
Combining the above equality with Lemma~\ref{lem:LuplusC2} yields that
for any $f\in H^{m}_+$ with $m=\lceil s\rceil$,
\[
\|Pf\|_{H^m}\leq C_3^m\|(L_{u_0}+C_2I)^mPf\|=C_3^m\|P(L_{u(t)}+C_2I)^mf\|=C_3^m\|(L_{u(t)}+C_2I)^mf\|\leq C_3^{2m}\|f\|_{H^m}.
\]
As $P$ is unitary, $P^{-1}=P^*$ is also bounded in $H^m$. According to Lemma~\ref{lem:interpolation},
\[
\|P\|_{H^s\to H^s}\leq \|P\|_{L^2\to L^2}^{(m-s)/m}\,\|P\|_{H^m\to H^m}^{s/m}\leq C_3^{2s}.
\]
Proceeding in the same way with $Q$, we obtain
\[
\|e^{-i\tilde t A}u^k\|_{H^s}\leq C_3^{2s}\|(S^*)^kQ\Pi u_0\|_{H^s}\leq C_3^{2s}\|Q\Pi u_0\|_{H^s}\leq C_3^{4s}\|\Pi u_0\|_{H^s}\leq C_3^{4s}\|u_0\|_{H^s}.
\]

For $u_0 \in H^s$ with $1<s< 2$, we can take a sequence $(u_0^n)_{n\in\mathbb N} \in (H^2)^{\mathbb N}$ that approximates $u_0$ in $H^s$. By the continuity of the flow map \cite[Theorem 1.1]{Mo07}, we have  $u^n(t) \underset{n \rightarrow \infty}{\longrightarrow} u(t)$ in $H^s$. Moreover, defining $A^n=I+2L_{u_0^n}$ and following the proof of Lemma~\ref{lem:diff-exp} we have that for every $v \in H_{+}^s$,
\begin{align*}
\|e^{itA^n}v-e^{itA}v\|_{H^s}
& \leq 2\int_0^t\|e^{i(t-t')A^n}(T_{u_0^n}-T_{u_0})e^{it'A}v\|_{H^s}dt'\\
& \leq 2\,C_3^{2s} \int_0^t\|(T_{u_0^n}-T_{u_0})e^{it'A}v\|_{H^s}dt'\\ & \leq 2\,C_1\,C_3^{4s}\,t\, \|u_0^n-u_0\|_{H^s}\|v\|_{H^s} \underset{n \rightarrow \infty}{\longrightarrow} 0.
\end{align*}

Hence, for fixed $t$ we have  $e^{itA^n} \underset{n \rightarrow \infty}{\longrightarrow} e^{itA}$ in $\mathcal{L}(H_{+}^s)$ (norm topology), and by applying an induction argument we obtain the convergence 
\[
e^{-i\tilde t A^n}(e^{itA^n}S^*)^k \Pi u_0^n \underset{n \rightarrow \infty}{\longrightarrow}  e^{-i\tilde t A}(e^{itA}S^*)^k\Pi u_0, \quad \text{in} \ H_{+}^s.
\]
By following the above proof with $u_0$ replaced by $u_0^n$, we have
\[
\|e^{-i\tilde t A^n}(e^{itA^n}S^*)^k \Pi u_0^n\|_{H^s} \leq (2C_1\max(\|u_0^n\|_{H^s},\|u^n(t)\|_{H^s})+2)^{4s}\|u_0^n\|_{H^s}.
\]
Therefore, by taking the limit as $n\rightarrow \infty$ in the above, we recover the desired bound by $C_3^{4s}\|u_0\|_{H^s}$ also in the case $1<s< 2$, which completes the proof.
\end{proof}

\begin{proof}[Proof of Theorem~\ref{thm:main}]
We recall that for notational convenience, we write $\|\cdot\|=\|\cdot\|_{L^2}$.
For $k\geq 0$, denote
\[
v^k=(e^{itA}S^*)^k\Pi u_0- (e^{itA_K}S^*)^k\Pi_K u_0,\quad \text{and} \quad w^k=e^{itA_K}S^*v^k.
\]
Notice that the Fourier coefficients of the error satisfy, for $k\in \{0,\dots,K-1\}$,
\[
e_k:=\widehat u(t,k) - \widehat{u_K}(t,k)
=\langle(e^{itA}S^*)^k\Pi u_0,1\rangle - \langle(e^{itA_K}S^*)^k\Pi_K u_0,1\rangle
=\langle v^k,1\rangle.
\]
Using the property of the shift operator $S^*$ and Lemma~\ref{lem:self-adjoint} thus yields
\begin{equation}
\label{eq:vkv0wk}
\|v^k\|^2=|\langle v^k,1\rangle|^2+\|S^*v^k\|^2
=|e_k|^2 +\|w^k\|^2.
\end{equation}
Moreover, recalling that $u^k=(e^{itA}S^*)^{k}\Pi u_0$ and defining
\[
\varepsilon^k=(e^{itA}-e^{itA_K})S^*u^k,
\]
we have
\begin{equation}
\label{eq:decompErr}
v^{k+1}=w^k+\varepsilon^{k}.
\end{equation}
For $C_4= 4\,C_1\|u_0\|_{H^s}^2C_3^{4s}$, we can bound
\[
\|\varepsilon^k\|=\|(e^{itA}-e^{itA_K})e^{-itA}u^{k+1}\| \leq 4C_1\|u_0\|_{H^s} tK^{-s}\sup_{t'\in [0,t]}\|e^{i(t'-t)A}u^{k+1}\|_{H^s} \leq C_4\,t\,K^{-s},
\]
where we used Lemma~\ref{lem:diff-exp} in the first inequality, and Lemma~\ref{lem:Hs-bound} in the second.

Applying successively \eqref{eq:decompErr} and \eqref{eq:vkv0wk}, we see that
\[
\|v^{k+1}\|\leq \|w^k\|+\|\varepsilon^{k}\| \leq \|v^k\|+\|\varepsilon^{k}\|.
\]
By induction, this implies that for all $k\ge 0$,
\[
\| v^k\|
\le\| v^0\| + \sum_{\ell =0}^{k-1} \|\varepsilon^\ell \| \le  (\|u_0\|_{H^s}+C_4\,t\,k)K^{-s}.
\]
Applying \eqref{eq:vkv0wk} and \eqref{eq:decompErr} one more time yields
\begin{align*}
\sum_{k=0}^{K-1} |e_k|^2 &= \sum_{k=0}^{K-1} \big[\|v^k\|^2-\|w^k\|^2\big]\\
&\leq \|v^0\|^2+\sum_{k=0}^{K-1} \big[\|v^{k+1}\|^2-\|w^{k}\|^2 \big]\\
&= \|v^0\|^2+\sum_{k=0}^{K-1} (\|v^{k+1}\|+\|w^{k}\|)(\|v^{k+1}\|-\|w^{k}\|)\\
& \leq \|v^0\|^2+\sum_{k=0}^{K-1} (\|v^{k+1}\|+\|v^{k}\|) \|\varepsilon^{k} \|\\
& \leq \|u_0\|_{H^s}^2K^{-2s}+2\sum_{k=0}^{K-1}  (\|u_0\|_{H^s}+C_4\,t\,K)\,C_4\,t\,K^{-2s}.
\end{align*}

As the coefficients with negative indices are just complex conjugates, we conclude that for $0\le r\le s$,
\begin{align*}
\|u-u_K\|_{H^r}^2
&\leq 2 \sum_{k \geq 0}(1+k^{2})^r|\widehat u(k)-\widehat{u_K}(k)|^2\\
& \leq 2K^{2r}\sum_{k=0}^{K-1}|e_k|^2 +\sum_{k\geq K}(1+k^{2})^r|\widehat u(k)|^2\\
& \leq C_5^2(1+tK)^2 K^{2r-2s},
\end{align*}
with 
\begin{equation}
\label{eq:cstFinal}
C_5=2\|u_0\|_{H^s}+ 2 C_4\leq 2^{8(s+1)^2}\max(\|u_0\|_{H^s},\|u(t)\|_{H^s},1)^{4s+1}\|u_0\|_{H^s}.
\end{equation}
\end{proof}

\begin{proof}[Proof of Theorem~\ref{thm:CS}]
The analogue of Proposition~\ref{lemma 4.2} and Proposition~\ref{lemma 4.3} for the \eqref{eq:CS} equation is obtained in \cite[Equation 2-11 and 2-14]{Bad24}.
By applying the same steps as in the proof of Theorem~\ref{thm:main} with the truncated Lax operator $L_{u_0,K}f=Df\mp \Pi_K(u_0\Pi_K(\overline{u_0}\Pi_K f))$, one  recovers the same convergence result for the scheme in equation \eqref{eq:newScheme} which approximates the \eqref{eq:CS} equation.
\end{proof}

\begin{rem}
The final result is actually slightly better than stated in Theorems~\ref{thm:main} and \ref{thm:CS}, since we achieve the optimal decay rate $K^{-s+r}$ for small times $t=\mathcal O(K^{-1})$. For small initial data, it is also readily seen that $C_5$ tends to $0$ linearly with $\|u_0\|_{H^s}$. 
\end{rem}

\section{Appendix}\label{sec:appendix}
\begin{proof}[Proof of Lemma~\ref{lem:bilinear}]
A proof on more general Sobolev spaces can be found in \cite[Theorem 4.39]{AF03}, here we present a much simpler argument in $H^s$.

For $k,\ell\in \mathbb Z$, denoting $\langle k\rangle=\sqrt{1+k^2}$, it holds $\langle k\rangle^\sigma\leq 2^{\sigma}(\langle\ell\rangle^\sigma+\langle k-\ell\rangle^\sigma)$. Letting $\mathcal D=(1+D^2)^{\sigma/2}$, this yields
\begin{align*}
\Big|\widehat{\mathcal D(fg)}(k)\Big|
=\langle k\rangle^\sigma\Big|\widehat{fg}(k)\Big|
=\langle k\rangle^\sigma\Big|\sum_{\ell\in\mathbb Z}\widehat f(\ell)\, \widehat g(k-\ell)\Big|
&\leq 2^{\sigma} \Big(\widehat{\mathcal Df}*\widehat g+\widehat f*\widehat{\mathcal Dg}\Big)(k).
\end{align*}
By Young's convolution inequality, for $p=\frac{2s}{2s-\sigma}$ and $q=\frac{2s}{s+\sigma}$, as $\frac{1}{p}+\frac{1}{q}=1+\frac{1}{2}$,
\[
\big\|\mathcal D(fg)\big\|=\big\|\widehat{\mathcal D(fg)}\big\|_2
\leq 2^{\sigma}\left(\big\|\widehat{\mathcal Df}\big\|_p\big\|\widehat g\big\|_q+\big\|\widehat f\big\|_1\big\|\widehat{\mathcal Dg}\big\|_2\right).
\]

Applying H\"older's inequality with exponents $\frac{2}{2-p}$ and $\frac{2}{p}$,
\[
\big\|\widehat{\mathcal Df}\big\|_p^p
=\sum_{k\in \mathbb Z}\langle k\rangle^{\sigma p} |\widehat f(k)|^p
\leq C_0^{\frac{2-p}{2}}\Big(\sum_{k\in \mathbb Z}\langle k\rangle^{2s}|\widehat f(k)|^2\Big)^{\frac{p}{2}}
=C_0^{\frac{2-p}{2}} \|f\|_{H^s}^p,
\]
where
$
C_0=\sum_{k\in \mathbb Z}\langle k\rangle^{-2s}.
$
By H\"older's inequality with exponents $\frac{2}{2-q}$ and $\frac{2}{q}$, and Cauchy-Schwarz inequality, we also have
\[
\big\|\widehat g\big\|_q^q\leq C_0^{\frac{2-q}{2}}\|g\|_{H^\sigma}^q
\quad \text{and}\quad
\|\widehat f\|_1\leq C_0^{\frac{1}{2}}\|g\|_{H^s}.
\]
Finally, as $\big\|\widehat{\mathcal Dg}\big\|_2= \|f\|_{H^\sigma}$, taking $C_1=2^{s+1}\sqrt{C_0}$, we conclude with
\[
\big\|\mathcal D(fg)\big\|
\leq 2^{\sigma}\Big(C_0^{\frac{1}{p}-\frac{1}{2}}C_0^{\frac{1}{q}-\frac{1}{2}}+C_0^\frac{1}{2}\Big)\|f\|_{H^s}\|g\|_{H^\sigma}
\leq C_1\|f\|_{H^s}\|g\|_{H^\sigma}.\qedhere
\]
\end{proof}

\begin{rem}
\label{remC1}
In the proof of the theorem, as $s\ge 1$, we use the bound
$C_0\leq \sum_{k\in \mathbb Z}\frac{1}{1+k^2}\leq 4$,
and thus $C_1\leq 2^{s+2}$.
\end{rem}

\begin{lem}
\label{lem:interpolation}
If $P$ is invertible in $H^m_+$ with $\|P\|_{L^2\to L^2}\leq 1$ and $\|P\|_{H^m\to H^m}\leq C_3^{2m}$, then
\[
\|P\|_{H^s\to H^s}\leq C_3^{2s}, \qquad 0 \le s \le m.
\]
\end{lem}
\begin{proof}
We only need to consider $0<s<m$. We use a simple version of $\mathcal{K}$ interpolation. A general proof can be found in \cite[Theorem~7.23]{AF03}. For $f\in H^s$, define
\begin{align*}
\mathcal{K}(t,f)&:=\inf_{g\in H^m}\|f-g\|^2+t\|g\|_{H^m}^2\\
&=\sum_{k\geq 0}\,\min_{\widehat g_k\in \mathbb C}\,|\widehat f_k- \widehat g_k|^2+t\langle k\rangle^{2m}|\widehat g_k|^2\\
&=\sum_{k\geq 0}|\widehat f_k|^2\min_{\lambda\in [0,1]}(1-\lambda)^2+t\langle k\rangle^{2m}\lambda^2\\
&=\sum_{k\geq 0}|\widehat f_k|^2\frac{t\langle k\rangle^{2m}}{1+t\langle k\rangle^{2m}}.
\end{align*}
Observing that for $0<s<m$,
\[
\int_0^\infty \frac{t\langle k\rangle^{2m}}{1+t\langle k\rangle^{2m}} \,\frac{dt}{t^{1+s/m}}=\langle k\rangle^{2s}\int_0^\infty \frac{x^{-s/m}}{1+x} \,dx=C_s\langle k\rangle^{2s},
\]
where $C_s$ only depends on $s$, we obtain
\[
\|f\|_{H^s}^2=\frac{1}{C_s}\int_0^\infty \mathcal{K}(t,f) \,\frac{dt}{t^{1+s/m}}.
\]
Finally, as $P$ is invertible in $H^m_+$,
\[
\mathcal{K}(t,Pf)=\inf_{g\in H^m_+}\|Pf-Pg\|^2+t\|Pg\|_{H^m}^2\leq \inf_{g\in H^m_+}\|f-g\|^2+C_3^{4m}t\|g\|_{H^m}^2= \mathcal{K}(C_3^{4m}t,f),
\]
and therefore
\[
\|Pf\|_{H^s}^2=\frac{1}{C_s}\int_0^\infty \mathcal{K}(t,Pf) \,t^{-\frac{s}{m}}\,\frac{dt}{t}\leq C_3^{4s}\|f\|_{H^s}^2.\qedhere
\]
\end{proof}


\begin{thebibliography}{99}

\bibitem{ABW} G. Abanov, E. Bettelheim, P. Wiegmann, {\it Integrable hydrodynamics of Calogero--Sutherland model: bidirectional Benjamin--Ono equation}, J. Phys. A 42 (2009), no.13, pp.135201.

\bibitem{AF03} R. Adams, J.J.F. Fournier, {\it Sobolev Spaces}, Springer, 2003.

\bibitem{AT91} C. Amick, J. Toland, {\it Uniqueness and related analytic properties for the Benjamin-Ono equation a nonlinear Neumann problem in the plane}, Acta Math., 167(1991), 107–126.

\bibitem{Bad24} R. Badreddine, {\it On the global well-posedness of the Calogero–Sutherland derivative nonlinear Schr\"odinger equation}, Pure and Applied analysis, Vol. 6 (2024), No. 2, p. 379-414.

\bibitem{Bad24.2} R. Badreddine, {\it Traveling waves and finite gap potentials for the Calogero--Sutherland derivative nonlinear Schrödinger equation}, Ann. Inst. H. Poincaré Anal. Non Linéaire 42 (2025), no. 4, pp. 1037–1092.

\bibitem{Bad24.3} R. Badreddine, {\it Zero dispersion limit of the Calogero-Moser derivative NLS equation}, SIAM Journal on Mathematical Analysis, 56(6), 7228-7249.

\bibitem{BO} T. Benjamin, {\it Internal waves of permanent form in fluids of great depth}, J. Fluid Mech., 29(1967), 559--592.

\bibitem{BE22} A. Biasi, O. Evnin, {\it Turbulent cascades in a truncation of the cubic Szeg\H o equation and related systems}, Analysis \& PDE, 15(1), 217-243, 2022.

\bibitem{BGGM24-1} E. Blackstone, L. Gassot, P. G\'erard, P. D. Miller, {\it The Benjamin-Ono Initial-Value Problem for Rational Data with Application to Long-Time Asymptotics and Scattering}, \href{https://arxiv.org/abs/2410.14870}{arXiv:2410.14870}.

\bibitem{BGGM24-2} E. Blackstone, L. Gassot, P. G\'erard, P. D. Miller, {\it The Benjamin–Ono equation in the zero-dispersion limit for rational initial data: generation of dispersive shock waves}, \href{https://arxiv.org/abs/2410.17405}{arXiv:2410.17405}.

\bibitem{BX11} J.P. Boyd, Z. Xu, {\it Comparison of three spectral methods for the Benjamin--Ono equation: Fourier pseudo-spectral, rational Christov functions and Gaussian radial basis functions}, Wave Motion 48, 702--706 (2011).

\bibitem{Ca71} F. Calogero, {\it Solution of the one-dimensional N-body problems with quadratic and/or inversely quadratic pair potentials}, Jour. of Math. Phys. 12 no. 3 (1971): 419--436.

\bibitem{CS22b} R. Carles, C. Su. {\it Scattering and uniform in time error estimates for splitting method in NLS}, Found. Comput. Math. 24 (2024), 683--722.

\bibitem{Chen} X. Chen, {\it Explicit formula for the Benjamin--Ono equation with square integrable and real valued initial data and applications to the zero dispersion limit},  Pure and Applied Analysis, Vol. 7 (2025), No. 1, 101–126.

\bibitem{C25} D. Clamond, {\it Explicit solution for the hyperbolic homogeneous scalar one-dimensional conservation law}, Appl. Math. Lett. (2025), p. 109593.

\bibitem{DA67} R. E. Davis and A. Acrivos, {\it Solitary internal waves in deep water}, J. Fluid Mech. 29 (1967), 593–607.

\bibitem{DM09} Z. Deng, H. Ma, {\it Optimal error estimates of the Fourier spectral method for a class of nonlocal, nonlinear dispersive wave equations}. Appl. Numer. Math. 59, 988--1010 (2009)

\bibitem{DM09-2} Z. Deng, H. Ma, {\it Error estimate of the Fourier collocation method for the Benjamin–Ono equation}. Numer. Math. Theor. Meth. Appl, 2(341-352), 1  (2009). 

\bibitem{DHK16} R. Dutta, H. Holden, U. Koley, N.H. Risebro, {\it Convergence of finite difference schemes for the Benjamin--Ono equation}. Numer. Math. 134, 249--274 (2016).

\bibitem{DS24} M. Dwivedi, T. Sarkar, {\it Stability of fully Discrete Local Discontinuous Galerkin method for the generalized Benjamin-Ono equation}. \href{https://arxiv.org/abs/2405.08360v6}{arXiv:2405.08360v6}.

\bibitem{G16} S. T. Galtung, {\it Convergence rates of a fully discrete Galerkin scheme for the Benjamin--Ono equation}, XVI International Conference on Hyperbolic Problems: Theory, Numerics, Applications. Cham: Springer International Publishing, 2016.

\bibitem{GGKM67} C.S. Gardner, J.M. Greene, M.D Kruskal, and R.M Miura. {\it Method for solving the Korteweg--de Vries equation}, Physical review letters, 19(19) :1095, 1967. doi:10.1103/PhysRevLett.19.1095.

\bibitem{Ga23.1} L. Gassot, Zero-dispersion limit for the Benjamin–Ono equation on the torus with single well initial data, Communications in Mathematical Physics, 401:2793--2843, 2023.

\bibitem{G} P. Gérard, {\it An explicit formula for the Benjamin--Ono equation}, Tunisian Journal of Mathematics, 2023, vol. 5, no 3, p. 593-603.

\bibitem{GG10} P. G\'erard and S. Grellier, {\it The cubic Szeg\H o equation}, Ann. Sci. Éc. Norm. Supér. (4), 43(5):761--810, 2010.

\bibitem{GG15} P. G\'erard and S. Grellier, {\it An explicit formula for the cubic Szeg\H o equation}, Trans. Amer. Math. Soc. 367 (2015), no. 4, 2979--2995.

\bibitem{GG17} P. Gérard and S. Grellier, {\it The cubic Szeg\H o equation and Hankel operators}, volume 389 of Astérisque. Soc. Math. de France, 2017.

\bibitem{GK21} P. Gérard, T. Kappeler, {\it On the integrability of the Benjamin--Ono equation on the torus}, Comm. Pure Appl. Math., 74 (2021), 1685--1747.

\bibitem{GKT} P. G\'erard, T. Kappeler, and P. Topalov, {\it Sharp well-posedness results of the Benjamin–Ono equation in $H^s(\mathbb{T},\mathbb{R})$  and qualitative properties of its solution}, Acta Mathematica, 231 :31–88, 2023.

\bibitem{GL22} P. G\'erard and E. Lenzmann, {\it The Calogero--Moser Derivative nonlinear Schr\"odinger equation}, Communications on Pure and Applied Mathematics, 77(10), 4008-4062.

\bibitem{GL24}  P. G\'erard and E. Lenzmann, {\it Global Well-Posedness and Soliton Resolution for the Half-Wave Maps Equation with Rational Data}, \href{https://arxiv.org/abs/2412.03351}{arXiv:2412.03351}.

\bibitem{GP24} P. G\'erard and A.~B. Pushnitski, The cubic Szeg\H o equation on the real line: explicit formula and well-posedness on the Hardy class, Communications in Mathematical Physics 405.7 (2024): 167.

\bibitem{HK24}
J. Hogan, M. Kowalski, {\it Turbulent threshold for continuum Calogero-Moser models}, 
Pure and Applied analysis, Vol. 6 (2024), No. 4, 941–954.

\bibitem{IT19} M. Ifrim, D. Tataru, {\it Well-posedness and dispersive decay of small data
solutions for the Benjamin–Ono equation}, Ann. Sci. \'Ec. Norm. Sup\'er. (4) 52 (2019), no. 2, 297--335.

\bibitem{KLV23} R. Killip, T. Laurens and M. Vi\c san. {\it Scaling-critical well-posedness for continuum Calogero-Moser models}, Commun. Am. Math. Soc. 5 (2025), 284-320.

\bibitem{KLV24} R. Killip, T. Laurens and M. Vi\c san. {\it  Sharp well-posedness for the Benjamin--Ono equation}, Inventiones mathematicae 236.3 (2024): 999-1054.

\bibitem{KKK24} K. Kim, T. Kim, S. Kwon, {\it Construction of smooth chiral finite-time blow-up solutions to Calogero--Moser derivative nonlinear Schrödinger equation},  to appear in Mem. Amer. Math. Soc. \href{https://arxiv.org/abs/2404.09603}{arXiv:2404.09603}.

\bibitem{KK24} T. Kim, S. Kwon, {\it Soliton resolution for Calogero--Moser derivative nonlinear Schrödinger equation}, \href{https://arxiv.org/abs/2408.12843}{arXiv:2408.12843}.

\bibitem{KS21} C. Klein and J.-C. Saut, {\it Nonlinear dispersive equations --- inverse scattering and PDE methods}, Applied Mathematical Sciences 209, Springer, Cham, 2021.

\bibitem{Lax68} P.D. Lax, {\it Integrals of nonlinear equations of evolution and solitary waves}, Comm. Pure Appl. Math. 21 (1968), 467--490.

\bibitem{MQ88} Y. Maday, A. Quarteroni, {\it Error analysis for spectral approximation of the Korteweg--de Vries equation}, Model. Math. Anal. Numer. 22 (3) (1988) 499--529.

\bibitem{Mo07} L. Molinet, {\it Global well-posedness in the energy space for the Benjamin–Ono equation on the circle}, Math. Ann., 337(2), 353-383, 2007.

\bibitem{Ono} H. Ono, Algebraic solitary waves in stratified fluids, J. Physical Soc. Japan 39 (1975), 1082--1091.

\bibitem{P11} O. Pocovnicu. Explicit formula for the solution of the Szeg\H o equation on the real line and applications. Discrete Contin. Dyn. Syst. A, 31(3) :607--649, (2011).

\bibitem{Pe95} D. E. Pelinovsky, {\it Intermediate nonlinear Schr\"odinger equation for internal waves in a fluid of finite depth}, Phys. Lett. A 197 (1995), no. 5--6, 401--406.

\bibitem{P24} M.O. Paulsen, {\it Justification of the Benjamin–Ono equation as an internal water waves model}, Ann. PDE. 10, 25, (2024), 1-129.

\bibitem{PD01} B. Pelloni, B., V.A. Dougalis, {\it Error estimate for a fully discrete spectral scheme for a class of nonlinear, nonlocal dispersive wave equations}. Appl. Numer. Math. 37, 95--107 (2001)

\bibitem{SunCS} R. Sun, {\it The intertwined derivative Schr\"odinger system of Calogero--Moser--Sutherland type.} Lett. Math. Phys. 114, 74. (2024). 

\bibitem{SunS} R. Sun, {\it The matrix Szeg\H o equation}, \href{https://arxiv.org/abs/2309.12136v1}{arXiv:2309.12136v1}.

\bibitem{Su71} B. Sutherland, {\it Exact results for a quantum many-body problem in one dimension}, Physical Review A 4, no.5 pp.2019 (1971).

\bibitem{Su75} B. Sutherland, {\it Exact ground-state wave function for a one-dimensional plasma}, Physical Review Letters, 34 no.17, pp.1083 (1975).

\bibitem{T98} V. Thomee, A.S.V. Murthy, {\it A numerical method for the Benjamin--Ono equation}, BIT 38(3), 597--611 (1998).

\end{thebibliography}
\end{document}